\documentclass[a4paper,11pt,reqno,noindent]{amsart}
\usepackage[centertags]{amsmath}
\usepackage{amsfonts,amssymb,amsthm,dsfont,cases,amscd,esint,enumerate,fancyhdr,newlfont}
\usepackage{comment}
\usepackage[english]{babel}
\usepackage[body={15cm,21.5cm},centering]{geometry} 
\usepackage{tikz}
\usepackage{blindtext}
\usepackage{mdframed}
\usepackage{float}
\usepackage{appendix}
\usepackage{ifthen}

\newcommand{\cited}[1]{%
	\ifthenelse{\equal{\@citeb{#1}}{}}{}{%
		\bibitem{#1}%
	}%
}

\usepackage{csquotes}
\usepackage[]{mdframed}
\usepackage{mathtools}
\usetikzlibrary{matrix}
\usepackage{hyperref}
\mathtoolsset{showonlyrefs}
\usepackage{yhmath}
\usetikzlibrary{positioning}
\usepackage{bm}
\usepackage{subfiles} 
\usepackage{tikz}
\usetikzlibrary{arrows.meta}

\newtheorem{theor0}{Theorem}[section]

\newenvironment{theor}
{\pushQED{\qed}\begin{theor0}}
	{\popQED\end{theor0}}
\newtheorem{lem0}[theor0]{Lemma}
\newenvironment{lem}
{\pushQED{\qed}\begin{lem0}}
	{\popQED\end{lem0}}
\newtheorem{prop0}[theor0]{Proposition}
\newenvironment{prop}
{\pushQED{\qed}\begin{prop0}}
	{\popQED\end{prop0}}
\newtheorem{cor0}[theor0]{Corollary}

\newtheorem{propr0}[theor0]{Property}

\newtheorem{hyp0}[theor0]{Hypothesis}

\newtheorem{result0}[theor0]{Result}

\newtheorem{conj0}[theor0]{Conjecture}

\newtheorem{heur0}[theor0]{Heuristics}

\theoremstyle{definition}
\newtheorem{defin0}[theor0]{Definition}
\newenvironment{defin}
{\pushQED{\qed}\begin{defin0}}
	{\popQED\end{defin0}}
\newtheorem{rems0}[theor0]{Remarks}

\newtheorem{ex0}[theor0]{Example}

\newtheorem{exs0}[theor0]{Examples}

\newtheorem{rem0}[theor0]{Remark}
\newenvironment{rem}
{\pushQED{\qed}\begin{rem0}}
	{\popQED\end{rem0}}
\newtheorem{qu0}[theor0]{Question}

\newtheorem{qus0}[theor0]{Questions}

\newtheorem{as0}[theor0]{Assumptions}

\theoremstyle{plain}

\numberwithin{equation}{section}

\newcommand{\rdash}{\mathcal{R}\text{ -}}
\newcommand{\N}{\mathbb N}

\newcommand{\R}{\mathbb R}

\usepackage{filecontents}

\title{An extension problem for higher order operators and operators of logarithmic type via renormalization}

\author[D. Lee]{David Lee}
\address[David Lee]{University of Technology Sydney, School of Mathematical and Physical Sciences, Sydney, Australia}
\email{davidchanwoo.lee@uts.edu.au}

\begin{document}
	
	\begin{abstract}
		We introduce a method of obtaining a higher order extension problem, \'a la Caffarelli-Silvestre, utilizing ideas from renormalization. Moreover, we give an alternative perspective of the recently developed extension problem for the logarithmic laplacian developed by Chen, Hauer and Weth (2023) [arXiv:2312.15689]. 
	\end{abstract}
	\maketitle

\section{Introduction}	

For a given $\sigma \in (0,1)$, consider the following Dirichlet problem
\begin{equation}\label{eqn:intro_Dirichlet_problem}
    \begin{cases}
    \displaystyle
        \nabla\cdot y^{1-2\sigma}\nabla u=0, \quad &\text{for $(x,y)\in \R^{d+1}_{+}:=\R^d\times (0,\infty),$}\\
      \displaystyle  u|_{y=0}=f, &\text{on $\mathbb{R}^{d}$},
    \end{cases}
\end{equation}
where $f \in \mathcal{S}(\R^d)$ (the space of Schwartz functions) $\nabla=(\partial_{x_1},...,\partial_{x_d},\partial_y)$ and $d\in \N$ (the set of natural numbers). If one considers the Dirichlet-to-Neumann map, given by 
\begin{equation}\label{eqn:DTN_map}
\begin{split}
    &f\in \mathcal{S}(\R^d)\mapsto -\lim_{y\rightarrow 0^{+}}y^{1-2\sigma}\partial_{y}u(\cdot,y),
\end{split}
\end{equation}
then one has that this map coincides with 
\begin{equation*}
\tfrac{\Gamma\left (1-\sigma\right )}{2^{2\sigma-1}}(-\Delta)^{\sigma},
\end{equation*}
where $(-\Delta)^{\sigma}$ is the fractional laplacian (which only depends on the horizontal variable $x$). This characterization is the famed Caffarelli-Silvestre extension \cite{MR2354493}. The Caffarelli-Silvestre extension has been revolutionary in recent years but one particular line of research that has developed from the original ideas of Caffarelli and Silvestre which relates to the following question:
\begin{center}
    \enquote{\emph{What other operators can be obtained via an extension problem?}}
\end{center}
There are several different perspectives on answering this question. One rather complete answer to this question is given by Kwa\'{s}nicki and Mucha \cite{MR3859452} where they show that one can obtain an extension problem for complete Bernstein functions of the negative laplacian. Moreover, Assing and Herman \cite{MR4262340} showed that the setting of Kwa\'{s}nicki and Mucha can be extended (for more general differential operators) using It\^o calculus. Nevertheless, the question of extension problems for higher order operators or operators of order \enquote{$0$} has also been a question of interest that doesn't quite fall into the framework of Kwa\'{s}nicki and Mucha \cite{MR3859452} \footnote{The Caffarelli-Silvestre extension \cite{MR2354493} or the extension problem of Kwa\'{s}nicki and Mucha \cite{MR3859452} can easily accomodate for higher order operators if one replaces the laplacian with powers of the laplacian. This is still a local extension but this setting isn't quite what we are interested in.}. An example of an extension problem of order \enquote{$0$} is the extension problem of the logarithmic laplacian due to Chen, Hauer and Weth \cite{chen2023extensionproblemlogarithmiclaplacian} which we will now briefly summarize. 

If one considers the following Neumann problem:
	\begin{equation}\label{eqn:intro_Neumann_problem_log}
	\begin{cases}
	\displaystyle
	\nabla\cdot y\nabla u=0, \quad &\text{for $(x,y)\in \R^{d+1}_{+},$}\\
	\displaystyle  -\lim_{y\rightarrow 0{+}}y\partial_{y}u=f, &\text{on $\mathbb{R}^{d}$},
	\end{cases}
	\end{equation}
    for $f\in \mathcal{S}(\R^d),$
it was shown by Chen, Hauer and Weth \cite{chen2023extensionproblemlogarithmiclaplacian} that the \emph{Neumann-to-weighted-Robin map}:
\begin{equation*}
    f\mapsto 2(\log(2)-\gamma_{E})f-2\lim_{y\rightarrow 0^{+}} (u(\cdot,y)+\log(y)f),
\end{equation*}
coincides with $\log(-\Delta)f$ ($\gamma_{E}$ is the Euler-Mascheroni constant). Here 
\begin{equation*}
    \log(-\Delta)f:=\lim_{\sigma\rightarrow 0^{+}}\frac{(-\Delta)^\sigma f-f}{\sigma}, \quad \text{for $f\in \mathcal{S}(\R^d),$}
\end{equation*}
which can also be understood, equivalently, in the sense of Fourier multipliers, cf. \cite{MR3995092}, \cite{MR209834} and \cite{MR1918790}. 

However, the boundary operator (the Neumann-to-weighted-Robin operator) that arises in the setting of Chen, Hauer and Weth \cite{chen2023extensionproblemlogarithmiclaplacian} probably doesn't come across as being the canonical boundary operator for the Neumann problem \eqref{eqn:intro_Neumann_problem_log}, this will be made more precise in Section \ref{sec_main_results}. The goal of this article is twofold:
\begin{itemize}
    \item to give an alternative perspective behind the Neumann-to-weighted-Robin operator that appears in the extension problem of Chen, Hauer and Weth \cite{chen2023extensionproblemlogarithmiclaplacian} by considering a reinterpretation through the lens of renormalization,
    \item provide a generalization of the result of Chen, Hauer and Weth \cite{chen2023extensionproblemlogarithmiclaplacian}, and in some sense Caffarelli and Silvestre \cite{MR2354493}, via the \emph{renormalized Neumann-to-Dirichlet} operator, see Definition \ref{def:renormalized_neumann}. 
\end{itemize}

Before we present our main result we briefly introduce some notation and conventions.

\subsection{Notation and conventions}	
\begin{itemize}
\item We denote $z^*$ to be the complex conjugate of $z$. 
\item We denote $\mathcal{S}(\R^d)$ to be set of Schwartz functions and $\mathcal{S}'(\R^d)$ to be the set of tempered distributions. Both are endowed with their usual topologies and we denote $\langle \cdot , \cdot \rangle_{\mathcal{S}'(\R^d),\mathcal{S}(\R^d)}$ to be the associated bracket for the pair $(\mathcal{S}'(\R^d),\mathcal{S}(\R^d)).$
	\item We denote the Fourier transform as 
	\begin{equation*}
	\hat{\varphi}(\xi)=\mathcal{F}(\varphi)(\xi):=\int_{\mathbb{R}^d}e^{i\xi\cdot x}\,\varphi(x)\,dx,\quad \text{for $\xi \in \R^d$,}
	\end{equation*}
	for $\varphi \in \mathcal{S}(\mathbb{R}^d)$. Similarly, we denote the Fourier inverse as 
	\begin{equation*}
	\check{\varphi}(x)=\mathcal{F}^{-1}(\varphi)(x):=(2\pi)^{-d}\int_{\mathbb{R}^d}e^{-i\xi\cdot x}\varphi(\xi)\,d\xi,\quad \text{for $\xi \in \R^d$.}
	\end{equation*}	
	\item We denote $\psi$ to be the usual digamma function and $\gamma_{E}$ to be the Euler-Mascheroni constant. Moreover, we remind the reader, cf. \cite{MR1388887}, that 
    \begin{equation}\label{eqn:upper_lower_bnd_psi}
        \log(z)-\frac{1}{z}\leq \psi(z)\leq \log(z)-\frac{1}{2z}, \quad \text{for $z>0.$}
    \end{equation}
 
    \item We will often denote $A\lesssim B$ to mean that there exists a constant $C>0$ such that $A\leq CB.$ Moreover, when we write $A_{\epsilon}\lesssim B_{\epsilon}$ then $A_{\epsilon}\leq CB_{\epsilon},$ where $C$ is independent of $\epsilon$. 

    \item As mentioned before, the gradient depends on both the horizontal and vertical variables $\nabla=(\partial_{x_1},...,\partial_{x_d},\partial_y)$ but the laplacian $\Delta=\sum_{i=1}^{d}\partial_{x_i}^2$ only depends on the horizontal variables.  
    \item We take the convention that $h(-\Delta)f:=\mathcal{F}^{-1}(h(|.|^2)\hat{f})$ for $f\in \mathcal{S}(\mathbb{R}^d)$. 
    \item We denote the modified Bessel function of the first kind $I_{\nu}$ which is given by

\begin{equation*}
	I_{\nu}(z):=\sum_{j=0}^\infty \frac{z^{2j+\nu}}{2^{2j+\nu}j!\Gamma(j+\nu+1)}, \quad \text{for $z>0$},
\end{equation*}
where $\nu \in \R$. The modified Bessel function of the second kind $K_{\nu}$ (also known as the Macdonald function) \footnote{In \cite{MR1411441}, it is referred to as the Macdonald function of the third kind.} is given as 
\begin{equation*}
	\begin{split}
		K_{\nu}:=\frac{\pi}{2\sin(\pi \nu)}\left (I_{-\nu}- I_{\nu}\right ), \quad &\text{for $\nu \in \mathbb{R}\setminus\mathbb{Z}$},\\
		K_{n}:=\lim_{\nu \rightarrow n}K_{\nu},\quad &\text{for $n\in \mathbb{Z}$}. 
	\end{split}
\end{equation*}
Moreover, we denote 
\begin{alignat*}{2}
    	&\hat{I}_{\nu}(z):=z^{-\nu}I_{\nu}(z), \quad &&\hat{K}_{\nu}(z):=z^{-\nu}K_{\nu}(z), \\
        &\tilde{I}_{\nu}(z):=z^{\nu}I_{\nu}(z),\quad &&\tilde{K}_{\nu}(z):=z^\nu K_{\nu}(z),\quad \text{for $\nu \in \R$ and $z>0$}.
\end{alignat*}
\end{itemize}

\subsection{Main Result}\label{sec_main_results}
Here, we consider the following Neumann problem:
	\begin{equation}\label{eqn:intro_Neumann_problem}
	\begin{cases}
	\displaystyle
	\nabla\cdot y^{1+2\nu}\nabla u=0, \quad &\text{for $(x,y)\in \R^{d+1}_{+},$}\\
	\displaystyle  -\lim_{y\rightarrow 0{+}}y^{1+2\nu}\partial_{y}u=f, &\text{on $\mathbb{R}^{d}$},
	\end{cases}
	\end{equation}
	where $f$ is a Schwartz function and $\nu >-1.$ One has that, if $\nu \in (-1,0)$ that the Neumann-to-Dirichlet map $f\mapsto u|_{y=0}$ coincides with $(-\Delta)^{2\nu}$ (the Riesz potential) up to a multiplicative constant. This inverse version of the Caffarelli-Silvestre extension has particular importance in the theory of Riesz gases, cf. references in Section \ref{subsec:riesz_energy}.

We now introduce the notion of solution for \eqref{eqn:intro_Neumann_problem}. 
    
	\begin{defin}\label{def:notion_soln}
		We say that $u\in W^{1,1}_{\text{loc}}(\R^{d+1}_{+};\, \,dx\,y^{1+2\nu}\,dy)$ if it is weakly differentiable on  $\R^{d+1}_{+}$ and both $u$ and $\nabla u$ are locally integrable with respect to the measure $\,dx\,y^{1+2\nu}\,dy$ on $\R^{d+1}_{+}$. 
		
		We say that $u$ is a distributional solution to \eqref{eqn:intro_Neumann_problem} if there exists 
        \begin{equation*}
        u\in W^{1,1}_{\text{loc}}(\R^{d+1}_{+};\, \,dx\,y^{1+2\nu}\,dy), 
        \end{equation*}
        such that 
        \begin{itemize}
            \item 
            \begin{equation*}
		\int_{0}^\infty\int_{\mathbb{R}^d}\nabla u(x,y)\cdot \nabla \varphi(x,y)\,dx\,y^{1+2\nu}\,dy=\int_{\R^d}f(x)\varphi(x,0)\,dx, \quad \text{for all $\varphi \in \mathcal{D}(\overline{\R^{d+1}_{+}})$},
		\end{equation*}
        where $\mathcal{D}(\overline{\R^{d+1}_{+}})$ is the set of smooth functions with compact support in $\overline{\R^{d+1}_{+}}$,
        \item for all $y>0$, we have that $u(\cdot ,y)\in W^{1,1}(\R^d)$, 
        \item $y\in (0,\infty)\mapsto \partial_{y}u(\cdot ,y)$ is continuous and integrable with respect to the horizontal variable $x.$
        \end{itemize}
        Moreover, we denote $u:=u_{f}$ to be the corresponding solution with Neumann boundary data $f$. 
	\end{defin}

From now, we focus our attention to the case when $\nu\geq 0.$ A naive choice for the boundary operator for this Neumann problem could simply be the Neumann-to-Dirichlet operator given by:
\begin{equation*}
    \Lambda_{\nu}:f\mapsto u_{f}|_{y=0}. 
\end{equation*}
Specifically, what we mean by canonical is that integration by parts would give us the following variational identity:
\begin{equation}\label{eqn:variational_iden_1}
    \int_{\R^d}(\Lambda_{\nu}f)f\,dx=\int_{0}^\infty \int_{\R^d}|\nabla u_{f}(x,y)|^2\,dx \,y^{1+2\nu}\,dy.
\end{equation}
This gives a relationship between the energy of system \eqref{eqn:intro_Neumann_problem} and the associated flux (which is reflected through the Dirichlet-Neumann pair). However, this only makes sense when $\nu \in (-1,0)$ otherwise both the left and right hand side takes the value of infinity when $f$ is non-zero. This is clear since, when $\nu \geq 0$, $u_{f}$ doesn't have a well-defined trace on $\{y=0\}$. In one sense, one can simply stop here but this is a setting that is familiar to those who are acquainted with Quantum Field theory. Through this lens, the natural question that we ask ourselves is whether one can \emph{remove the appropriate infinite part} and consider the operator corresponding to the \emph{renormalized energy} of \eqref{eqn:intro_Neumann_problem}. To be more precise about this, we now introduce the notion of the renormalized limit and Hadamard regularization. 

\begin{defin}\label{def:hadamard}
    Suppose that $\Psi:(0,1)\rightarrow \mathbb{C}$ is a function such that there exists $N\in \N$ and constants $a_k,\lambda_{k}\in \mathbb{C}$ where 
    \begin{itemize}
        \item $\text{Re}(\lambda_{k})>0$ for $k=1,..,N$,
        \item     \begin{equation}
        \Psi(z)=a_0\log(z^{-1})+\sum_{k=1}^N a_k z^{-\lambda_k}+\chi(z), \quad \text{for $z\in (0,1)$,}
    \end{equation}
    where $\lim_{z\rightarrow 0^{+}}\chi(z)$ exists and is finite. 
    \end{itemize}
Then, we say that 
\begin{equation*}
    \rdash\lim_{z\rightarrow 0^{+}}\Psi(z):=\lim_{z\rightarrow 0^{+}}\chi(z). 
\end{equation*}
We call $\rdash\lim_{z\rightarrow 0^{+}}\Psi(z)$ the \emph{renormalized limit} of $\Psi$ at $0$. 
\end{defin}

\begin{defin}
    Let $h:\R^d\setminus\{0\}\rightarrow \mathbb{C}$ be integrable on the set $\{x\in \R^d:|x|>\epsilon\}$ for all $\epsilon \in (0,1)$. Then, we denote 
    \begin{equation*}
        \int_{\R^d}^\text{Had}h(x)\,dx:=\rdash\lim_{\epsilon\rightarrow 0^{+}} \int_{|x|>\epsilon}h(x)\,dx. 
    \end{equation*}
    $\int_{\R^d}^\text{Had}h(x)\,dx$ is known as the \emph{Hadamard regularization} of the integral $\int_{\R^d}h(x)\,dx$. Similarly, if $h:(0,\infty)\rightarrow \mathbb{C}$ is locally integrable on the set $\{x\in\R:x>\epsilon\}$, for all $\epsilon \in (0,1)$, then we denote 
    \begin{equation*}
        \int_{0}^{\infty,\text{Had}}h(x)\,dx:=\rdash\lim_{\epsilon\rightarrow 0^{+}} \int_{\epsilon}^\infty h(x)\,dx. 
    \end{equation*}
    Again, $\int_{0}^{\infty,\text{Had}}h(x)\,dx$ is known as the Hadamard regularization of $\int_{0}^{\infty}h(x)\,dx$. 
\end{defin}
Hadamard regularization is fundamental within the theory of renormalization. It gives us a way of assigning a finite value to integrals which would otherwise take the value of infinity. It should be remarked that Marcel Riesz justified this definition by showing that Hadamard finite part coincides with the meromorphic continuation of a convergent integral, cf. the following articles by Riesz \cite{MR30102, MR147616} or the book of Paycha \cite{MR2987296}. This is what lead to the notion of Riesz fractional differentiation/the fractional laplacian. 

To quickly summarize, how we define the boundary operator of \eqref{eqn:intro_Neumann_problem} is by considering the following: 
\begin{itemize}
    \item we want a operator that naturally arises from integration by parts so that we have an identity that is analogous to \eqref{eqn:variational_iden_1}, 
    \item this is not possible due to the right hand side of \eqref{eqn:variational_iden_1} taking the value $\infty$, 
    \item instead we consider the operator that arises via a renormalized version of integration by parts. 
\end{itemize}

This leads us to the following definition.
\begin{defin}\label{def:renormalized_neumann}
Let $\nu \geq 0.$ We define the \emph{renormalized Neumann-to-Dirichlet operator} $\Lambda_{\nu}^{\mathcal{R}}:\mathcal{S}(\R^d)\rightarrow \mathcal{S}'(\R^d)$ as 

    	 \begin{equation*}
     \langle \Lambda^{\mathcal{R}}_{\nu}f,g\rangle_{\mathcal{S}'(\R^d),\mathcal{S}(\R^d)}=\int_{0}^{\infty,\text{Had}} \int_{\R^d} \nabla u_{f}(x,y) \cdot \nabla u_{g}(x,y)\,dx \,y^{1+2\nu}\,dy, \quad \text{for $f,g\in \mathcal{S}(\R^d).$}
 \end{equation*}
\end{defin}
It should be noted that whilst the above definition relies on a renormalized integration by parts one could say, on a formal level, 
\begin{equation*}
    \Lambda_{\nu}^\mathcal{R}f=\rdash\lim_{y\rightarrow 0^{+}}u_{f}(\cdot,y). 
\end{equation*}

We now present the main result of this article. 

\begin{theor}\label{thm:main_result}   
We have the following classification of the \emph{renormalized Neumann-to-Dirichlet} operator
$\Lambda^{\mathcal{R}}_{\nu}:\mathcal{S}(\mathbb{R}^d)\rightarrow \mathcal{S}'(\mathbb{R}^d)$ given by:
\begin{equation}
        \Lambda_{\nu}^\mathcal{R}f=
    \begin{cases}
     \displaystyle\frac{\pi}{2^{\nu+1}\Gamma(\nu+1)\sin(\pi\nu)}(-\Delta)^{\nu}f,\quad &\text{for $\nu \in \mathbb{R}_{+}\setminus\mathbb{N}_0,$}\\
     \displaystyle\frac{(-1)^\nu}{2^\nu\nu!}\left ( \psi(1)+\psi(\nu+1)+2\log(2)-\log(-\Delta)\right )(-\Delta)^{\nu}f,&\text{for $\nu \in \mathbb{N}_{0}.$}
\end{cases}
\end{equation}
\end{theor}

The case $\nu \in \mathbb{N}_0$ could be viewed as surprising however this situation is typical in Quantum Field theory, or renormalization, as a \emph{anomaly phenomenon}. What we mean is that the homogeneous property, in the case when $\nu \in \mathbb{R}_{+}\setminus\mathbb{N}_0$, is not extended to the case when $\nu \in \mathbb{N}_0$, c.f. \cite{derezi2023generalized} or \cite{MR1918790}. We will look more deeply in the case when $\nu=0$ in the upcoming Section \ref{sec:renorm}. 

Given that the above result gives rise to higher order fractional laplacians it is natural for us to briefly discuss on the existing extension methods which also give rise to such operators. 

\begin{itemize}
    \item \textbf{Modification of the boundary operator}\ \\
    The following was first observed by Chang and Gonz\'alez \cite{MR2737789}. 
    
    Let $\sigma\in (0,\infty)\setminus \N$ and let $u$ be the bounded solution of \eqref{eqn:intro_Dirichlet_problem}. Moreover, we denote $[\sigma]$  to be the integer part of $\sigma$. In this case, the Dirichlet-to-Neumann map, given by \eqref{eqn:DTN_map}, is not well-defined. Nevertheless, one can obtain the following classification
    \begin{equation}\label{eqn:chang_gonz_neumann}
        \lim_{y\rightarrow 0^{+}}y^{1-2(\sigma-[\sigma])}\partial_{y}\left ( \tfrac{2}{y}\partial_{y}\right )^{[\sigma]}u(\cdot,y)=\frac{(-1)^{[\sigma]+1}\Gamma([\sigma]+1-\sigma)}{4^{\sigma-([\sigma]+1/2)}\Gamma(\sigma)}(-\Delta)^\sigma f,
    \end{equation}
    for all $f\in \mathcal{S}(\R^d).$
    \item \textbf{Higher order extension problem}\ \\
    As observed in Yang \cite{MR3592161} one instead could consider a higher order extension problem. We present a version that is a special case of a result by Biswas and Stinga \cite{MR4742773} although there are other versions due to \cite[Cora and Musina]{MR4429579} and \cite[Musina and Nazarov]{MR4735195}. 

    Let $\sigma\in (0,\infty)\setminus \N$ and let $u$ be the bounded solution of \eqref{eqn:intro_Dirichlet_problem} of the following higher order extension problem 
    \begin{equation*}
    \begin{cases}
    \displaystyle
        \left ((-\Delta)u-\tfrac{1-2(\sigma-[\sigma])}{y}\partial_{y}-\partial_{yy}\right)^{[\sigma]+1}u=0, \,\,&\text{for $(x,y)\in \R^{d+1}_{+},$}\\
      \displaystyle\lim_{y\rightarrow 0^{+}}y^{1-2(\sigma-[\sigma])}\partial_{y}\left ((-\Delta)u-\tfrac{1-2(\sigma-[\sigma])}{y}\partial_{y}-\partial_{yy}\right)^{m}u=0, &\text{on $\mathbb{R}^{d}$},\\
      \displaystyle\lim_{y\rightarrow 0^{+}}y^{1-2(\sigma-[\sigma])}\left ((-\Delta)u-\tfrac{1-2(\sigma-[\sigma])}{y}\partial_{y}-\partial_{yy}\right)^{m}u=c_{m,\sigma}(-\Delta)^mf, &\text{on $\mathbb{R}^{d}$},
    \end{cases}
\end{equation*}
where the last 2 equalities are satisfied for integers $0\leq m<[\sigma]$ and 
\begin{equation*}
    c_{m,\sigma}=\frac{[\sigma]!\Gamma(\sigma-m)}{(\sigma-m)!\Gamma(\sigma)}, \quad \text{for $0\leq m<[\sigma]$}. 
\end{equation*}
One can observe that the corresponding boundary mapping for the above problem can be expressed as 
    \begin{equation*}
    \begin{split}
        \displaystyle&\lim_{y\rightarrow 0^{+}}y^{1-2(\sigma-[\sigma])}\partial_{y}\left ((-\Delta)u-\tfrac{1-2(\sigma-[\sigma])}{y}\partial_{y}-\partial_{yy}\right)^{[\sigma]}u\\
&=\frac{(-1)^{[\sigma]+1}\Gamma([\sigma]+1-\sigma)[\sigma]!}{4^{\sigma-([\sigma]+1/2)}\Gamma(\sigma)}(-\Delta)^\sigma f.
    \end{split}
    \end{equation*}

\end{itemize}

The extension problems above avoid the case when $\sigma \in \N$. The author is unaware if analogous results, to those in Theorem \ref{thm:main_result}, can be obtained for the case when $\sigma \in \N_0$ but it would be interesting to see whether this is possible.

\begin{rem}
    We should remark that the ideas of renormalization were already present in Yang \cite{yang2013higherorderextensionsfractional} within the context of the extension problem. Nevertheless, the approach taken by Yang is quite different from the one that we consider here. 
\end{rem}

\begin{rem}
The result of Chang and Gonz\'alez \cite{MR2737789} uses a rather beautiful insight from scattering theory and hyperbolic geometry. Nevertheless, the author strongly suspects that this result can be reinterpreted through the lens of Quantum Field theory. It would be interesting to see if such a result can be obtained. 
\end{rem}

\subsection{Outline of the article}

We summarize the rest of this article. 

\begin{itemize}
    \item In Section \ref{sec:renorm}, we give some context behind our result by comparing ideas within the theory of Coulomb gases and the extension problem of the logarithmic laplacian. 
     \item In Section \ref{sec:macd}, we review some brief facts on the Macdonald function. 
     \item In Section \ref{sec:main}, we prove the main result.
\end{itemize}

\section{Some context behind our result}\label{sec:renorm}

The topic of renormalization has seen a significant resurgence in recent years, largely driven by Hairer's development of the theory of regularity structures \cite{MR3274562} and singular stochastic partial differential equations. Similarly, singular integral operators have garnered renewed interest, particularly due to the extension problem of Caffarelli and Silvestre \cite{MR2354493}. Despite sharing common mathematical roots, the connection between these two areas seems to have been overlooked in recent years. For example, the fractional laplacian (and fractional differentiation in general) arises naturally from Hadamard regularization, cf. Samko \cite{MR1918790}. The goal of this section is to look at the direct sources of inspiration for this article which, we hope, can give the reader a better appreciation of our result. The main sources of inspiration are:

\begin{itemize}
    \item the theory of Coulomb/Riesz gases, 
    \item the extension problem for the logarithmic laplacian. 
\end{itemize}

\subsection{The modulated energy for Coulomb gases}
\label{subsec:riesz_energy}
The recent development of theory of Coulomb/Riesz gases, cf. the books of Serfaty \cite{serfaty2024lecturescoulombrieszgases, MR3309890}, is one of the main inspirations for this work. Whilst this theory is largely independent from the rest of the article, we believe that it is worthwhile to present this to the reader to get some better context behind the definition of the renormalized Neumann-to-Dirchlet map given that its definition directly comes from the notion of the \emph{modulated energy}. Very similar ideas are also present in the theory of Ginzburg-Landau vortices, cf. the book of Bethuel, Brezis and H\'elein \cite{MR3618899}. We now introduce some Hamiltonians.

We denote $g:\R^d\setminus\{0\}\rightarrow \R$ to be the (rescaled) Coulomb kernel given by 
\begin{equation*}
        g(x):=\begin{cases}
        -\log|x|, \quad &\text{for $d=2$,}\\
            |x|^{-(d-2)}, \quad &\text{for $d\geq 3$,}
        \end{cases}
\end{equation*}
and we consider the following Hamiltonian:
\begin{equation*}
    H_{n}(x_1,...,x_n):=\sum_{i\ne j}^{n}g(x_i-x_j)+n\sum_{i=1}^{n}V(x_i),\quad \text{for $x_i$ distinct,}
\end{equation*}
and the following energy functional on $ \mathcal{P}(\R^d)$ (the set of probability measures on $\R^d$):
\begin{equation*}
    \begin{split}
        \mathcal{E}:&\mathcal{P}(\R^d)\rightarrow (-\infty,\infty],\\
    \mathcal{E}(\mu):&=\int_{\R^d}\int_{\R^d} g(x-y)\mu(dx)\mu(dy)+\int_{\R^d}V(x)\,\mu(dx), \quad \text{for $\mu \in \mathcal{P}(\R^d),$}
    \end{split}
\end{equation*}
where $V$ satisfies the following:
\begin{itemize}
    \item lower semi-continuous, 
    \item bounded below, 
    \item $\lim_{|x|\rightarrow \infty}V(x)=\infty$,
    \item the set $\{V<\infty\}$ has positive $g-$capacity.  
\end{itemize}

Frostman's theorem \cite{MR36370} gives a very complete picture on the properties of the minimizer of $\mathcal{E}$. In particular, we know that 

\begin{itemize}
    \item there exists a unique minimizer of $\mathcal{E}$ (we denote this as $\mu_V$),
    \item the support of $\mu_V$ is bounded and has positive $g-$capacity,
    \item for $h^{\mu_V}:=g*\mu_V$, we have that 
    \begin{equation}
    \begin{cases}
        h^{\mu_{V}}+\frac{1}{2}V&=c, \quad \text{on the support of $\mu_V$,}\\
        h^{\mu_{V}}+\frac{1}{2}V&\geq c, \quad \text{otherwise,}
    \end{cases}
    \end{equation}
    where $c:=\mathcal{E}(\mu_{V})-\frac{1}{2}\int_{\R^d}V(x)\,\mu_{V}(dx)$. 
\end{itemize}

However, there are some natural questions to be considered from this. For instance:
\begin{itemize}
    \item Does the minimizer of $H_n$ converge to the minimizer of $\mathcal{E}$?
    \item What is the asymptotic behaviour of $\min_{n}H_n$ as $n\rightarrow \infty$? 
    \item What configuration of points minimizes $H_n$?
\end{itemize}
The answers to these other questions have been recent developments due to Serfaty, Petrache and others, cf. \cite{MR3353821, MR3646281,MR4099791}. To answer these questions, we need to introduce the notion of the \emph{modulated energy}. 

 The modulated energy arises from the following observation. Consider the following \emph{splitting formula} which can be obtained by direct computation:
\begin{equation*}
\begin{split}
        H_{n}
    &=n^2\mathcal{E}(\mu_V)+2n\sum_{i=1}^n\left ( \int_{\R^d}g(x_i-y)\,d\mu_V(y)+\tfrac{1}{2}V(x_i)-c\right )\\
    &+\int \int_{D^c} g(x-y)\,d(\nu_{n}-n\mu_V)(x)\,d(\nu_{n}-n\mu_V)(y), 
\end{split}
\end{equation*}
where $D$ is the diagonal ie $\{x=y\}$. Here the very last term is the most interesting and it is natural for us to consider a \emph{carr\'e du champ} type representation which is useful for implementing screening type arguments, cf. \cite[Section 7.2]{MR3353821}. If we denote 
\begin{equation*}
    h_n:=g*(\nu_n-n\mu_{V}), 
\end{equation*}
or equivalently $h_n:=c_{d}(-\Delta)^{-1}(\nu_n-\mu_{V})$ for some suitable constant $c_{d}$, then on a formal level we have that 
\begin{equation*}
\begin{split}
    \int \int_{D^c} g(x-y)\,d(\nu_{n}-n\mu_V)(x)\,d(\nu_{n}-n\mu_V)(y)
    &=\int_{\R^d}h_n(y)\,d(\nu_n-n\mu_{V})(dy)\\
    &\approx \frac{1}{c_d}\int_{\R^d}h_n(y)(-\Delta)h_n(y)\,dy,\\
    &\approx \frac{1}{c_d}\int_{\R^d}|\nabla h_n(y)|^2\,dy. 
\end{split} 
\end{equation*}
This naive application of Green's theorem is definitely not valid given that the value on the right hand side is infinite (due to the presence of the dirac measures) and the value on the left hand side is finite. It would be more reasonable (although still not really valid) to say that 
\begin{equation*}
    \int_{\R^d} \int_{\R^d} g(x-y)\,d(\nu_{n}-n\mu_V)(x)\,d(\nu_{n}-n\mu_V)(y)=\frac{1}{c_d}\int_{\R^d}|\nabla h_n(y)|^2\,dy.
\end{equation*}
Nevertheless, it isn't equal to the value of $\int \int_{D^c} g(x-y)\,d(\nu_{n}-n\mu_V)(x)\,d(\nu_{n}-n\mu_V)(y)$ which is the term of interest. What we could do is we can minus the infinite part (ie renormalize) and obtain the correct identity. Namely, by carefully applying integration by parts and integrating away from the singular points, we have the following identity, from Sandier and Serfaty \cite{MR3353821}, which can be obtained from a renormalized integration by parts:

\begin{equation}\label{eqn:kwas_dirichlet}
     \begin{split}
         &\int \int_{D^c} g(x-y)\,d(\nu_{n}-n\mu_V)(x)\,d(\nu_{n}-n\mu_V)(y)\\
     &=\lim_{\eta \rightarrow 0^{+}}\left (\int_{\R^d\setminus \cup_{i=1}^{n}B(x_i,\eta)}|\nabla h_n(y)|^2\,dy-ng(\eta) \right ).
     \end{split} 
\end{equation}
We are now at a point where we can explain to the reader the significance of renormalization in this setting. One has an identity that is looks to be true on a formal level however a correction of infinite order (Hadamard Regularization) is often needed not just to make sense of the objects that we consider but also to give us the correct result \eqref{eqn:kwas_dirichlet}. Moreover, we can see that by applying a renormalized integration by parts we can obtain the \enquote{correct} identity. 

The definition of the \emph{renormalized Neumann-to-Dirichlet} operator, see Definition \ref{def:renormalized_neumann}, essentially takes direct inspiration from the previous computations as it is natural boundary term that arises via a renormalized integation by parts applied to \eqref{eqn:intro_Neumann_problem}. 

\subsection{Alternative perspectives on the logarithmic extension problem of Chen, Hauer and Weth}

Here we give 2 different perspectives on the logarithmic extension problem of Chen, Hauer and Weth \cite{chen2023extensionproblemlogarithmiclaplacian} which can be viewed as a special case of Theorem \ref{thm:main_result} ($\nu=0$). Both approaches are, in some sense, typical from the perspective of Quantum Field theory especially when one considers the notion of dimensional regularization. The computations here are not rigourous. Nevertheless, we feel that this gives the reader a good sense as to what to expect for the proof of Theorem \ref{thm:main_result}. 

\begin{itemize}
    \item \textbf{An approach via an asymptotic expansion of the modified Bessel function of the second kind}\ \\

In Theorem 1.2 of \cite{chen2023extensionproblemlogarithmiclaplacian} one has that the solution of \eqref{eqn:intro_Neumann_problem_log} can be expressed as 
\begin{equation*}
    u_{f}(x,y)=\frac{\pi^{d/2}}{2\Gamma(d/2)}\int_{\R^d}\frac{f(z)}{(|x-z|^2+y^2)^{d/2}}\,dz, \quad \text{for $(x,y)\in\R^d\times (0,\infty).$}
\end{equation*}
If one takes the Fourier transform with respect to the horizontal variables (in the sense of tempered distributions) then one has that 
\begin{equation}\label{eqn:log_fourier_exp}
    \hat{u}_{f}(\xi,y)=K_0(|\xi|y)\hat{f}(\xi), \quad \text{for $y>0$ and $\xi \in \R^d\setminus\{0\},$}
\end{equation}
where $K$ is the modified Bessel function of the second kind. 

By integration by parts and the Parseval identity, we have that 
\begin{equation*}
\begin{split}
    \int_{\epsilon}^\infty \int_{\R^d}\nabla u_{f}(x,y)\cdot \nabla u_{g}(x,y)\,dx\,y\,dy&=\int_{\R^d}u_{f}(x,\epsilon)(-\epsilon \partial_{y}u_{g}(x,\epsilon))\,dx,\\
    &=\int_{\R^d}\hat{u}_{f}(\xi,\epsilon)(-\epsilon \partial_{y}\hat{u}_{g}(\xi,\epsilon))^{*}\,d\xi,
\end{split}
\end{equation*}
for $\epsilon \in (0,1).$ 

Assuming we can bring the renormalized limit inside the integral, 
\begin{equation}\label{eqn:problem}
    \begin{split}
            &\int_{0}^{\infty, \text{Had}} \int_{\R^d}\nabla u_{f}(x,y)\cdot \nabla u_{g}(x,y)\,dx\,y\,dy\\
            &=\rdash\lim_{\epsilon \rightarrow 0^{+}}\int_{\R^d}\hat{u}_{f}(\xi,\epsilon)(-\epsilon \partial_{y}\hat{u}_{g}(\xi,\epsilon))^{*}\,d\xi\\
            &=\int_{\R^d}\left (\rdash\lim_{\epsilon \rightarrow 0^{+}}\hat{u}_{f}(\xi,\epsilon) \right ) \left (\lim_{\epsilon\rightarrow 0^{+}}-\epsilon \partial_{y}\hat{u}_{g}(\xi,\epsilon))\right )^{*}\,d\xi\\
            &=\int_{\R^d}\left (\rdash\lim_{\epsilon \rightarrow 0^{+}}\hat{u}_{f}(\xi,\epsilon) \right ) \hat{g}(\xi)^{*}\,d\xi.
    \end{split}
\end{equation}

By considering the characterization of \eqref{eqn:log_fourier_exp}, and the following asymptotic expansion (near $0$), we have that 
\begin{equation*}
    K_0(\epsilon|\xi|)\hat{f}(\xi)\sim (\log(2)-\gamma_{E})\hat{f}(\xi)-\frac{1}{2}\log(|\xi|^2)\hat{f}(\xi)-\log(\epsilon)\hat{f}(\xi)+O(\epsilon^2)
\end{equation*}
From this, we have that 
\begin{equation*}
\begin{split}
    &\rdash\lim_{\epsilon \rightarrow 0^{+}}\hat{u}_{f}(\xi,\epsilon)\\
    &=\lim_{\epsilon \rightarrow 0^{+}}\hat{u}_{f}(\xi,\epsilon)+\log(\epsilon)\hat{f}(\xi),\\
    &=(\log(2)-\gamma_{E})\hat{f}(\xi)-\frac{1}{2}\log(|\xi|^2)\hat{f}(\xi), \quad \text{for $\xi \in \R^d\setminus\{0\},$}
\end{split}
\end{equation*}
which in turn gives us 
\begin{equation*}
    \Lambda_{0}^{\mathcal{R}}f=(\log(2)-\gamma_{E})f-\frac{1}{2}\log(-\Delta)f, \quad \text{for $f\in \mathcal{S}(\R^d).$}
\end{equation*}

Note that this operator is slightly different from the one considered by Chen, Hauer and Weth \cite{chen2023extensionproblemlogarithmiclaplacian}. This is because we prioritize the renormalization process when considering the choice of the boundary operator. Nevertheless, we only differ by a linear transformation. 

\begin{rem}
    We really want to stress to the reader that the above computation is not rigorous and should only be used as a formal computation. For instance, in \eqref{eqn:problem}, we assumed that 
\begin{equation*}
    \rdash\lim_{z\rightarrow 0^{+}}a(z)b(z)=\left (\rdash\lim_{z\rightarrow 0^{+}}a(z)\right )\left (\lim_{z\rightarrow 0^{+}}b(z)\right). 
\end{equation*}
This is definitely not true in general (take $a(z)=z^{-1}$ and $b(z)=z)$. 
\end{rem}

\item \textbf{An approach leveraging computed renormalized bilinear integrals}\ \\
Here we leverage recent results due to Derezi\'nski, Gaß and Ruba \cite{derezi2023generalized}. In particular, \cite[Proposition 3.1]{derezi2023generalized} which states:
\begin{equation}\label{eqn:renormalized_bilinear}
\begin{split}
        &b^2\int_{0}^{\infty,\text{Had}}|K_{\nu}(by)|^2y\,dy\\
    &=\frac{(-1)^{\nu}}{2}(1+\nu\log\left ( \frac{b^2}{4}\right )+2\nu(1-\psi(1+\nu))), \quad \text{for $b>0$ and $\nu \in \N_0.$}
\end{split}
\end{equation}
We now show the reader the relevance of the above result in regards to computing the renormalized Neumann-to-Dirichlet operator $\Lambda_{0}^\mathcal{R}$. By the Parseval identity, we have that 
\begin{equation*}
\begin{split}
    &\int_{\epsilon}^\infty \int_{\R^d}\nabla u_{f}(x,y)\cdot \nabla u_{g}(x,y)\,dx\,y\,dy\\
    &=\int_{\epsilon}^\infty \int_{\R^d}|\xi|^2\hat{u}_{f}(\xi,y)\hat{u}_{g}(\xi,y)^*\,d\xi\,y\,dy\\
    &\qquad \qquad +\int_{\epsilon}^\infty \int_{\R^d}\partial_{y}\hat{u}_{f}(\xi,y)\partial_{y}\hat{u}_{g}(\xi,y)^*\,d\xi\,y\,dy,\\
    &=\int_{\epsilon}^\infty \int_{\R^d}|\xi|^2|K_{0}(|\xi|y)|^2\left (\hat{f}(\xi)\hat{g}(\xi)^*\right )\,d\xi\,y\,dy\\
    &\qquad\qquad +\int_{\epsilon}^\infty \int_{\R^d}|\xi|^2|K_{1}(|\xi|y)|^2\left (\hat{f}(\xi)\hat{g}(\xi)^*\right )\,d\xi\,y\,dy,\\
    &= \int_{\R^d}\left (|\xi|^2\int_{\epsilon}^\infty|K_{0}(|\xi|y)|^2\,y\,dy\right )\left (\hat{f}(\xi)\hat{g}(\xi)^*\right )\,d\xi\\
    &\qquad\qquad +\int_{\R^d}\left (|\xi|^2\int_{\epsilon}^\infty |K_{1}(|\xi|y)|^2\,y\,dy\right )\left (\hat{f}(\xi)\hat{g}(\xi)^*\right )\,d\xi,
\end{split}
\end{equation*}
for $\epsilon \in (0,1).$ Assuming we can switch the renormalized limit and the integral, we have, from \eqref{eqn:renormalized_bilinear}, the following: 
\begin{equation*}
\begin{split}
    &\int_{0}^{\infty,\text{Had}} \int_{\R^d}\nabla u_{f}(x,y)\cdot \nabla u_{g}(x,y)\,dx\,y\,dy\\
    &=\int_{\R^d}\left ((\log(2)-\gamma_{E}-\frac{1}{2}\log(|\xi|^2))\hat{f}(\xi)\right )\hat{g}(\xi)^*\,d\xi,\\
    &=\int_{\R^d}\left ((\log(2)-\gamma_{E}-\frac{1}{2}\log(-\Delta))f(x)\right )g(x)\,dx.
\end{split}
\end{equation*}
Hence, we have that 
\begin{equation*}
    \Lambda_{0}^{\mathcal{R}}f=(\log(2)-\gamma_{E})f-\frac{1}{2}\log(-\Delta)f, \quad \text{for $f\in \mathcal{S}(\R^d).$}
\end{equation*}

\end{itemize}

\begin{rem}
    Even though we presented the case of $\nu=0$ a similar process can be done for the case when $\nu\geq 0$. 
\end{rem}

\section{A quick review on the modified Bessel function of first and second kind}

\label{sec:macd}

The modified Bessel functions of the first and second kind play a significant role in this work. Most of the facts presented here are standard and can be found 
in the following references \cite[Appendix]{MR1912205}, \cite[Section 1.2]{MR1411441} or \cite[page 290]{MR374877}. However, we will present these for the convenience of the reader. 

We recall that, cf. \cite[page 290]{MR374877},
\begin{equation}\label{eqn:bessel_lim}
    \begin{split}
        |K_{\nu}(z)|&\lesssim \frac{e^{-z}}{z^{1/2}},\quad \text{for $z>1$},\nu \in\R,\\
        \lim_{z\rightarrow 0^{+}}\tilde{K}_{\nu}(z)&=\frac{\Gamma(\nu)}{2^{\nu+1}}, \quad \text{for $\nu>0$},\\
    \end{split}
\end{equation}

This shows that $\tilde{K}_{\nu}$ is a bounded function which decays exponentially. Moreover, it has a continuous extension to $0$, for $\nu >0.$

It is clear that $\hat{K}_\nu$ has very singular behaviour near $0$. For this reason, it will be useful for us to consider a Laurent series type representation \footnote{For the integer case we don't quite have a Laurent series.} of $\hat{K}_{\nu}$. For ease of presentation, we will denote the function $C_{\nu}:(0,\infty)\rightarrow \R$, which is given by 
\begin{equation}\label{eqn:Cnu}
C_{\nu}(z):=
\begin{cases}\displaystyle
-\frac{\pi}{2\sin(\pi\nu)}\sum_{j=0}^{\lfloor \nu\rfloor }  \frac{z^{-2(\nu-j)}}{2^{2j-\nu}j!\Gamma(j-\nu+1)}, \quad &\text{for $\nu \in \R_{+}\setminus \N_0$},\\
\displaystyle\mathds{1}_{\nu\ne 0}\sum_{j=0}^{\nu-1} \frac{(-1)^j(\nu-j-1)!}{2^{2j-\nu+1}j!} z^{-2(\nu-j)},\quad &\text{for $\nu \in\N_0$},
\end{cases}
\end{equation}
and $A_{\nu}:(0,\infty)\rightarrow \R$, which is given by 
\begin{equation}\label{eqn:Anu}
\begin{split}
&A_{\nu}(z)\\
:=
&\begin{cases}
\displaystyle
\sum_{j=\lfloor \nu \rfloor +1}^\infty
\frac{\pi}{2\sin(\pi\nu)}\left (\frac{z^{2(j-\nu)}}{2^{2j-\nu}j!\Gamma(j-\nu+1)}\right )\,\,&\text{for $\nu \in \R_{+}\setminus \N_0$},\\
\displaystyle\qquad-\sum_{j=1}^\infty \frac{\pi}{2\sin(\pi\nu)}\left (\frac{z^{2j}}{2^{2j+\nu}j!\Gamma(j+\nu+1)} \right ), &\\
\displaystyle \sum_{j=1}^\infty \left (\frac{\psi(j+1)+\psi(j+\nu+1)+2\log(2)}{2^{2j+\nu}j!(\nu+j)!}\right )z^{2j},\quad&\text{for $\nu \in\N_0$},\\
\displaystyle\qquad -\sum_{j=1}^\infty \left (\frac{1}{2^{2j+\nu-1}j!(\nu+j)!}\right )z^{2j}\log(z).  &
\end{cases}
\end{split}
\end{equation}	
The following proposition will justify our choice of functions.  

\begin{prop}\label{prop:series}
	We have the following series expansions for $\hat{K}_{\nu}$:
	\begin{list}{$\circ$}{}
		\item $\nu \in \R_{+}\setminus\N_0$
		\begin{equation*}
		\hat{K}_{\nu}(z)-C_{\nu}(z)=\frac{\pi}{2^{\nu+1}\Gamma(\nu+1)\sin(\pi\nu)}+A_{\nu}(z), \quad \text{for $z>0$}.
		\end{equation*}	
		\item $\nu \in \N_0$
		\begin{equation*}
		\hat{K}_{\nu}(z)-C_{\nu}(z)=\frac{(-1)^\nu}{2^\nu\nu!}\left ( \psi(1)+\psi(\nu+1)+2\log(2)-2\log(z)\right )+A_{\nu}(z), \quad \text{for $z>0$}.
		\end{equation*} 
	\end{list}
\end{prop}
\begin{proof}
	
	We split the cases when $\nu$ is an integer and not an integer.  
	
	\begin{list}{$\circ$}{}
	\item $\nu \in \mathbb{R}_{+}\setminus\mathbb{N}_0$ 
	
	In this case, we can verify that 
	\begin{equation*}
		\hat{K}_{\nu}(z)=\frac{\pi}{2\sin(\pi\nu)}\sum_{j=0}^\infty \left ( \frac{z^{2(j-\nu)}}{2^{2j-\nu}j!\Gamma(j-\nu+1)} -\frac{z^{2j}}{2^{2j+\nu}j!\Gamma(j+\nu+1)}\right ), \quad \text{for $z>0$}. 
	\end{equation*}
	By splitting the sums appropriately, we have that 
	\begin{equation*}
	\begin{split}
			\hat{K}_{\nu}(z)
			=&\frac{\pi}{2\sin(\pi\nu)}\sum_{j=0}^\infty \left ( \frac{z^{2(j-\nu)}}{2^{2j-\nu}j!\Gamma(j-\nu+1)} -\frac{z^{2j}}{2^{2j+\nu}j!\Gamma(j+\nu+1)}\right ),\\
			=&\frac{\pi}{2\sin(\pi\nu)}\sum_{j=0}^{\lfloor \nu\rfloor}  \frac{z^{-2(\nu-j)}}{2^{2j-\nu}j!\Gamma(j-\nu+1)}\\
			&-\frac{\pi}{2^{\nu+1}\Gamma(\nu+1)\sin(\pi\nu)}\\ 
			&+
			\sum_{j=\lfloor \nu \rfloor +1}^\infty
			\frac{\pi}{2\sin(\pi\nu)}\left (\frac{z^{2(j-\nu)}}{2^{2j-\nu}j!\Gamma(j-\nu+1)}\right )\\
            &-\sum_{j=1}^\infty \frac{\pi}{2\sin(\pi\nu)}\left (\frac{z^{2j}}{2^{2j+\nu}j!\Gamma(j+\nu+1)}\right ).
	\end{split}
	\end{equation*}
	\item $\nu \in \mathbb{N}_0$ 
	
	Using a characterization that appears in \cite[Section 3.2]{derezi2023generalized}, we have that for $\nu \in \N_0$ and $z>0$:
		\begin{equation*}
		\begin{split}
		\hat{K}_{\nu}(z)
		=&\mathds{1}_{\nu\ne 0}\sum_{j=0}^{\nu-1} \frac{(-1)^j(\nu-j-1)!}{2^{2j-\nu+1}j!} z^{-2(\nu-j)}\\
		&+(-1)^\nu\sum_{j=0}^\infty \left (\frac{\psi(j+1)+\psi(j+\nu+1)+2\log(2)-2\log(z)}{2^{2j+\nu}j!(\nu+j)!}\right )z^{2j},\\
				=&\mathds{1}_{\nu\ne 0}\sum_{j=0}^{\nu-1} \frac{(-1)^j(\nu-j-1)!}{2^{2j-\nu+1}j!} z^{-2(\nu-j)}\\
				&+\frac{(-1)^\nu}{2^\nu\nu!}\left ( \psi(1)+\psi(\nu+1)+2\log(2)-2\log(z)\right )\\
				&+\sum_{j=1}^\infty \left (\frac{\psi(j+1)+\psi(j+\nu+1)+2\log(2)-2\log(z)}{2^{2j+\nu}j!(\nu+j)!}\right )z^{2j}.
					\end{split}
	\end{equation*}
	From this, we have our result. 
\end{list}
\end{proof}

\section{Proof of the main result}\label{sec:main}

We give a brief preview of the main ingredients to the reader our approach of proving Theorem \ref{thm:main_result}. 

Our first goal will be to derive the Fourier transform, with respect to the horizontal variable $x$, of the solution of \eqref{eqn:intro_Neumann_problem}. Specifically, in Section \ref{subsec:solving_neumann}, we will show that if $u$ solves \eqref{eqn:intro_Neumann_problem}, then we have the following characterization of the Fourier transform of $u_{f}$ with respect to the horizontal variable (which we denote as $\hat{u}_{f}$).

\begin{prop}\label{prop:soln}
	Let $u_{f}$ be a solution to \eqref{eqn:intro_Neumann_problem} corresponding to having the Neumann boundary condition $f\in \mathcal{S}(\R^d)$. Then necessarily, we have that 
	\begin{equation}\label{eqn:preview_soln}
	\hat{u}_{f}(\xi,y)=\frac{1}{2^\nu\Gamma(1+\nu)}\hat{K}_{\nu}(|\xi|y)\left (|\xi|^{2\nu}\hat{f}(\xi) \right), \quad \text{for almost all $\xi \in \R^{d}\setminus \{0\}$ and $y>0$}. 
	\end{equation}
\end{prop}

    Proving the above proposition will be done by reducing the Neumann problem to the following Sturm-Liouville problem given here:

    \begin{equation}\label{eqn:Sturm-Liouville}
\begin{dcases}
\lambda\phi(y)-\frac{1}{y^{1+2\nu}}\frac{d}{dy}\left ( y^{1+2\nu}\phi'(y)\right )&=0, \quad \text{for $y>0$},\\
\hspace{2.1cm}-\lim_{y\rightarrow {0}^{+}}y^{1+2\nu}\phi'(y)&=1,\quad \lim_{y\rightarrow \infty}\phi(y)=0,
\end{dcases}
\end{equation}
where $\lambda>0$ and $\nu >-1$.

From this expression, we have the following which can be obtained via an application of integration by parts and the Plancherel identity:

	\begin{equation*}
	\begin{split}
	&\int_{\epsilon}^\infty \int_{\R^d} \nabla u_{f}(x,y) \cdot \nabla u_{g}(x,y)\,dx \,y^{1+2\nu}\,dy\\
	&=\int_{\R^d} u_{f}(x,\epsilon) \left (-\epsilon^{1+2\nu} (\partial_{y}u_{g})(x,\epsilon)\right )\,dx,\\
	&=\int_{\R^d} \hat{u}_{f}(\xi,\epsilon) \left (-\epsilon^{1+2\nu} (\partial_{y}\hat{u}_{g})(\xi,\epsilon)\right )^*\,d\xi, 
	\end{split}
	\end{equation*}
    for $\epsilon \in (0,1).$

From the expression given in \eqref{eqn:preview_soln}, we have that:

\begin{equation*}
\begin{split}
\hat{u}_{f}(\xi,\epsilon)
&=\frac{1}{2^\nu\Gamma(1+\nu)}\hat{K}_{\nu}(|\xi|\epsilon)\left (|\xi|^{2\nu}\hat{f}(\xi) \right ),\\
\left (-\epsilon^{1+2\nu} (\partial_{y}\hat{u}_{g})(\xi,\epsilon)\right )&=\frac{1}{2^\nu\Gamma(1+\nu)}\tilde{K}_{1+\nu}(|\xi|\epsilon)\hat{g}(\xi),
\end{split}
\end{equation*}
which gives us
\begin{equation}\label{eqn:prior_renormalize}
\begin{split}
    &\int_{\epsilon}^\infty \int_{\R^d} \nabla u_{f}(x,y) \cdot \nabla u_{g}(x,y)\,dx \,y^{1+2\nu}\,dy\\
    &=\left (\frac{1}{2^\nu\Gamma(1+\nu)}\right )^2\int_{\R^d}\hat{K}_{\nu}(|\xi|\epsilon)\tilde{K}_{1+\nu}(|\xi|\epsilon)\left (|\xi|^{2\nu}\hat{f}(\xi)\hat{g}(\xi)^* \right )\,d\xi,
\end{split}
\end{equation}
for $\epsilon\in (0,1)$. To prove the main result, we simply have to compute the renormalized limit of the right hand side of \eqref{eqn:prior_renormalize}. In order to do this we will need to anticipate the singular behaviour of $\hat{K}_{\nu}$. By letting Proposition \ref{prop:soln_sturm} guide us, we have following expressions:
\begin{equation}\label{eqn:renormalize_decomp}
\begin{split}
    &\hat{K}_{\nu}(|\xi|\epsilon)-C_{\nu}(|\xi|\epsilon)=\frac{\pi}{2^{\nu+1}\Gamma(\nu+1)\sin(\pi\nu)}+A_{\nu}(|\xi|\epsilon),\quad \text{for $\nu \in \R_{+}\setminus\N_0$,}\\
       &\hat{K}_{\nu}(|\xi|\epsilon)-C_{\nu}(|\xi|\epsilon)-\frac{(-1)^\nu}{2^{\nu-1}\nu!}\log(\epsilon^{-1})\\
            &=\frac{(-1)^\nu}{2^\nu\nu!}\left ( \psi(1)+\psi(\nu+1)+2\log(2)-\log(|\xi|^2)\right )+A_{\nu}|\xi|\epsilon),\quad\text{for $\nu \in \N_0$.}
\end{split}
\end{equation}

With this in mind, we denote 
\begin{equation}\label{eqn:cnu_tilde}
    \tilde{C}_{\nu}(\xi,\epsilon):=C_{\nu}(|\xi|\epsilon)-\frac{(-1)^\nu}{2^{\nu-1}\nu!}\log(\epsilon^{-1})\mathds{1}_{\nu \in \N_0}, \quad \text{for $\xi \in \R^d\setminus \{0\},$ and $\epsilon\in(0,1),$}
\end{equation}
and we re-express the right hand side of \eqref{eqn:prior_renormalize} as 
\begin{equation}\label{eqn:series_laurent}
    \begin{split}
        &\left (\frac{1}{2^\nu\Gamma(1+\nu)}\right )^2\int_{\R^d}\hat{K}_{\nu}(|\xi|\epsilon)\tilde{K}_{1+\nu}(|\xi|\epsilon)\left (|\xi|^{2\nu}\hat{f}(\xi)\hat{g}(\xi)^* \right )\,d\xi\\
    =&\left (\frac{1}{2^\nu\Gamma(1+\nu)}\right )^2\int_{\R^d}\left (\hat{K}_{\nu}(|\xi|\epsilon)-\tilde{C}_{\nu}(\xi,\epsilon) \right )\tilde{K}_{1+\nu}(|\xi|\epsilon)\left (|\xi|^{2\nu}\hat{f}(\xi)\hat{g}(\xi)^* \right )\,d\xi\\
    &+\left (\frac{1}{2^\nu\Gamma(1+\nu)}\right )^2\int_{\R^d}\tilde{C}_{\nu}(\xi,\epsilon)\tilde{K}_{1+\nu}(|\xi|\epsilon)\left (|\xi|^{2\nu}\hat{f}(\xi)\hat{g}(\xi)^* \right )\,d\xi.
    \end{split}
\end{equation}
The goal of Section \ref{subsec:renormal_lim} will be to compute the limit of the first integral which we now state here. 
\begin{prop}\label{prop:renormalized_lim}
We have that 
\begin{itemize}
    \item     
    Assume that $\nu \in \R_{+}\setminus\N_0$. Then, we have that 
    \begin{equation*}
    \begin{split}
    &\lim_{\epsilon\rightarrow 0^{+}}\left (\frac{1}{2^\nu\Gamma(1+\nu)}\right )^2\int_{\R^d}\left (\hat{K}_{\nu}(|\xi|\epsilon)-\bar{C}_{\nu}(\xi,\epsilon)\right )\tilde{K}_{1+\nu}(|\xi|\epsilon)\left (|\xi|^{2\nu}\hat{f}(\xi)\hat{g}(\xi)^* \right )\,d\xi\\
    &=\frac{\pi}{2^{\nu+1}\Gamma(\nu+1)\sin(\pi\nu)}\int_{\R^d}|\xi|^{2\nu}\hat{f}(\xi)\hat{g}(\xi)^* \,d\xi. 
    \end{split}
    \end{equation*}
    \item 
    Assume that $\nu \in \N_0$. Then, we have that
    
    \begin{equation*}
          \begin{split}
              &\lim_{\epsilon\rightarrow 0^{+}}\left (\frac{1}{2^\nu\Gamma(1+\nu)}\right )^2\int_{\R^d}\left (\hat{K}_{\nu}(|\xi|\epsilon)-\bar{C}_{\nu}(\xi,\epsilon)\right )\tilde{K}_{1+\nu}(|\xi|\epsilon)\left (|\xi|^{2\nu}\hat{f}(\xi)\hat{g}(\xi)^* \right )\,d\xi\\
              &=\frac{(-1)^\nu}{2^\nu\nu!}\int_{\R^d}\left ( \psi(1)+\psi(\nu+1)+2\log(2)-\log(|\xi|^2)\right )\left (|\xi|^{2\nu}\hat{f}(\xi)\hat{g}(\xi)^*\right ) \,d\xi.
          \end{split}
    \end{equation*}
    \end{itemize}
\end{prop}

In Section \ref{subsec:renormal_lim_2}, we will compute the renormalized limit of the second integral which we state here.

\begin{prop}\label{prop:renormalized_lim_1}
     We have that 
\begin{equation*}
    \mathcal{R}\lim_{\epsilon\rightarrow 0^{+}}\left (\frac{1}{2^\nu\Gamma(1+\nu)}\right )^2\int_{\R^d}\tilde{C}_{\nu}(\xi,\epsilon)\tilde{K}_{1+\nu}(|\xi|\epsilon)\left (|\xi|^{2\nu}\hat{f}(\xi)\hat{g}(\xi)^* \right )\,d\xi=0. 
\end{equation*}
\end{prop}

With these propositions we are now in a position to prove the main result. 
\begin{proof}[Proof of Theorem \ref{thm:main_result}]
    From the decomposition given in \eqref{eqn:series_laurent} and Proposition \ref{prop:renormalized_lim} and \ref{prop:renormalized_lim_1} we must have that 
    \begin{equation}
        \begin{split}
            &\int_{0}^{\infty,\text{Had}} \int_{\R^d} \nabla u_{f}(x,y) \cdot \nabla u_{g}(x,y)\,dx \,y^{1+2\nu}\,dy\\
            &=\begin{cases}
                \displaystyle\frac{\pi}{2^{\nu+1}\Gamma(\nu+1)\sin(\pi\nu)}\int_{\R^d}|\xi|^{2\nu}\hat{f}(\xi)\hat{g}(\xi)^* \,d\xi, &\text{$\nu \in \mathbb{R}_{+}\setminus\N_0,$}\\
               \displaystyle\frac{(-1)^\nu}{2^\nu\nu!}\int_{\R^d}\left ( \psi(1)+\psi(\nu+1)+2\log(2)-\log(|\xi|^2)\right )\left (|\xi|^{2\nu}\hat{f}(\xi)\hat{g}(\xi)^*\right ) \,d\xi,&\text{$\nu \in \mathbb{N}_{0}.$}
            \end{cases}
        \end{split}
    \end{equation}
    By the Plancherel identity, we obtain our result. 
\end{proof}

\subsection{Solving the Neumann problem} \label{subsec:solving_neumann}

In this section, we will solve the Neumann problem problem \eqref{eqn:intro_Neumann_problem} by reducing it to the following Sturm-Liouville problem via the Fourier transform.

We now introduce the appropriate notion of solution of \eqref{eqn:Sturm-Liouville}. 

\begin{defin0}
	We say that $\phi$ is a solution to \eqref{eqn:Sturm-Liouville} if:
	\begin{enumerate}
		\item $\phi\in W^{1,1}_{\text{loc}}((0,\infty);\,y^{1+2\nu}\,dy)$, that is $\phi$ and its weak derivative is locally integrable with respect to the measure $\mathds{1}_{(0,\infty)}\,y^{1+2\nu}\,dy$,
		\item for all $\tau\in C^{\infty}_{c}([0,\infty))$, we have that 
		\begin{equation}
		\lambda\int_{0}^\infty \phi(y)\tau(y)y^{1+2\nu}\,dy+\int_{0}^\infty \phi'(y)\tau'(y)y^{1+2\nu}\,dy =\tau(0),
		\end{equation}
		\item 
		\begin{equation}
 \lim_{y\rightarrow \infty}\phi(y)=0.
		\end{equation}
	\end{enumerate}
\end{defin0}

We now derive an explicit expression of the solution of \eqref{eqn:Sturm-Liouville}. 
\begin{lem}\label{prop:soln_sturm}
Let $\phi$ be a solution of \eqref{eqn:Sturm-Liouville}. Then, for $\nu>-1$, we have that 
	\begin{equation*}
	\phi(y)=\frac{\lambda^{\nu}}{2^\nu\Gamma(1+\nu)}\hat{K}_{\nu}(\sqrt{\lambda}y), \quad \text{for $y>0$}.
	\end{equation*} 
\end{lem}

\begin{proof}
	If one ignores the boundary conditions it is well known that this Sturm-Liouville problem has 2 linearly independent solutions. Specifically, from \cite[page 654, equation 11]{MR1503439}, we know that:
	\begin{equation*}
	\phi(y)=A\hat{I}_{\nu}(\sqrt{\lambda}y)+B\hat{K}_{\nu}(\sqrt{\lambda}y), \quad \text{for $y>0$},
	\end{equation*}
	
	We need to identify what the constants $A$ and $B$ must be. 
	
	First observe that the condition $\lim_{y\rightarrow \infty}\phi(y)=0$ implies that $A$ must necessarily be $0$. 

	We now compute the precise value of $B$. Since $\frac{d}{dy}\hat{K}_{\nu}(y)=-y^{-\nu}K_{1+\nu}(y)$ for $y>0$, cf. \cite[page 638]{MR1912205}. In turn, we have that 
	\begin{equation*}
	\begin{split}
	\phi'(y)=-B\sqrt{\lambda}(\sqrt{\lambda}y)^{-\nu}K_{1+\nu}(\sqrt{\lambda}y),
	\end{split}
	\end{equation*}
    Hence, 
    \begin{equation*}
        -y^{1+2\nu}\phi'(y)=B\lambda^{-\nu}\tilde{K}_{1+\nu}(\sqrt{\lambda}y), \quad \text{for $y>0.$}
    \end{equation*}
    Since $\lim_{y\rightarrow 0^{+}}\tilde{K}_{1+\nu}(x)=\Gamma(1+\nu)2^{\nu}$, from \eqref{eqn:bessel_lim}, we have that 
    \begin{equation*}
       1= -\lim_{y\rightarrow 0^{+}}y^{1+2\nu}\phi'(y)=B\lambda^{-\nu} \Gamma(1+\nu)2^{\nu}
    \end{equation*}
    In turn, we obtain our result. 

\end{proof}
We now prove Proposition \ref{prop:soln}. 

\begin{proof}[Proof of Proposition \ref{prop:soln}]
	By considering \eqref{def:notion_soln} we proceed by making a appropriate choice of $\varphi$. First, note that 
	\begin{equation*}
		\begin{split}
			&\int_{0}^\infty\int_{\mathbb{R}^d} \nabla u_{f}(x,y)\cdot \nabla \varphi(x,y)\,dx\,y^{1+2\nu}\,dy\\
			&=		\int_{0}^\infty\int_{\mathbb{R}^d} u_{f}(x,y)(-\Delta_{x}) \varphi(x,y)\,dx\,y^{1+2\nu}\,dy+\int_{0}^\infty\int_{\mathbb{R}^d} \partial_{y}u_{f}(x,y)\partial_{y}\varphi(x,y)\,dx\,y^{1+2\nu}\,dy\\
			&=\int_{\R^d}f(x)\varphi(x,0)\,dx,
		\end{split}
	\end{equation*}
	for all $\varphi \in \mathcal{D}(\overline{\R^{d+1}_{+}})$. Choosing $\varphi(x,y)=\hat{v}(x)\tau(y)$, where $\tau\in C_{c}^\infty([0,\infty))$ and $v\in \mathcal{S}(\R^d)$ such that $\hat{v}\in \mathcal{D}(\R^d)$, we have that 
		\begin{equation*}
	\begin{split}
&\int_{0}^\infty\left (\int_{\mathbb{R}^d} u_{f}(x,y)(-\Delta_{x}) \hat{v}(x)\,dx \right )\,\tau(y)y^{1+2\nu}\,dy+\int_{0}^\infty\left (\int_{\mathbb{R}^d} \partial_{y}u_{f}(x,y)\hat{v}(x)\,dx\right )\tau'(y)\,y^{1+2\nu}\,dy\\
	&=\tau(0)\left (\int_{\R^d}f(x)\hat{v}(x)\,dx \right ). 
	\end{split}
	\end{equation*}
	We now focus on re-expressing each integral that are within the brackets. 
	
	\begin{itemize}
		\item We focus on the following integral:
		\begin{equation*}
			\int_{\mathbb{R}^d} u_{f}(x,y)(-\Delta_{x}) \hat{v}(x)\,dx. 
		\end{equation*}
		By definition, we have that 
		\begin{equation*}
		\begin{split}
			\int_{\mathbb{R}^d} u_{f}(x,y)(-\Delta_{x}) \hat{v}(x)\,dx&= 		\int_{\mathbb{R}^d} u_{f}(x,y)(-\Delta_{x}) \left ( \int_{\mathbb{R}^d} e^{ix\cdot \xi}v(\xi)\right )\,dx,\\
			&=		\int_{\mathbb{R}^d} u_{f}(x,y) \int_{\mathbb{R}^d} e^{ix\cdot \xi}|\xi|^2v(\xi)\,\xi\,dx. 
		\end{split}
		\end{equation*}
		Moreover, utilizing the assumption that $u_{f}(\cdot,y)\in W^{1,1}(\R^d)$ for all $y>0$, we have by an application of the Fubini theorem that 
		\begin{equation*}
			\begin{split}
							\int_{\mathbb{R}^d} u_{f}(x,y)(-\Delta_{x}) \hat{v}(x)\,dx&=\int_{\mathbb{R}^d} u_{f}(x,y) \left ( \int_{\mathbb{R}^d} e^{ix\cdot \xi}|\xi|^2v(\xi)\,d\xi\right )\,dx,\\
							&=\int_{\mathbb{R}^d} \left (|\xi|^2\hat{u}_{f}(\xi,y)\right )v(\xi)\,d\xi. 
			\end{split}
		\end{equation*} 
		\item We now consider the integral
		\begin{equation*}
		\int_{\mathbb{R}^d} \partial_{y}u_{f}(x,y)\hat{v}(x)\,dx.
		\end{equation*}
		By the assumptions that and $u_{f}(\cdot, y)\in W^{1,1}(\R^d)$ for $y>0$ and $y\in (0,\infty)\mapsto \partial_{y}u_{f}(\cdot,y)$ is continuous, we have that 
		\begin{equation}
		\begin{split}
			\int_{\mathbb{R}^d} \partial_{y}u_{f}(x,y)\hat{v}(x)\,dx&=	\partial_{y}\left (\int_{\mathbb{R}^d} u_{f}(x,y)\hat{v}(x)\,dx\right ),\\
			&=\int_{\mathbb{R}^d} \partial_{y}\hat{u}_{f}(\xi,y)v(\xi)\,d\xi.
		\end{split}
				\end{equation}
	\item For the last integral, we have that 
	\begin{equation*}
	\int_{\R^d}f(x)\hat{v}(x)\,dx=\int_{\R^d}\hat{f}(\xi)v(\xi)\,d\xi,
	\end{equation*}
	by an application of the Fubini theorem. 
		
	\end{itemize}
		With this in consideration, we have that 
		\begin{equation*}
		\begin{split}
			&\int_{\mathbb{R}^d} 
			\left (
			|\xi|^2\int_{0}^\infty \hat{u}_{f}(\xi,y) \tau(y) y^{1+2\nu}\,dy+\int_{0}^\infty\partial_{y}\hat{u}_{f}(\xi,y) \tau'(y)y^{1+2\nu}\,dy \right ) v(\xi)\,d\xi\\
			&=\int_{\R^d}\left (\tau(0)\hat{f}(\xi)\right )v(\xi)\,d\xi. 
		\end{split}
		\end{equation*}	
		From Proposition \ref{prop:soln_sturm} and considering that $v$ is arbitrary, we must that 
		\begin{equation}
			|\xi|^2\int_{0}^\infty \hat{u}_{f}(\xi,y) \tau(y) y^{1+2\nu}\,dy+\int_{0}^\infty\partial_{y}\hat{u}_{f}(\xi,y) y^{1+2\nu}\tau'(y)\,dy=\tau(0)\hat{f}(\xi),
		\end{equation}
		for almost every $\xi \in \R^d\setminus\{0\}$. Hence, we obtain the result. 
\end{proof}
\subsection{Proof of Proposition \ref{prop:renormalized_lim}}\label{subsec:renormal_lim}

We begin by taking our focus on computing the limits that appear in Proposition \ref{prop:renormalized_lim}.     
\begin{proof}[Proof of Proposition \ref{prop:renormalized_lim}]
We now compute 
\begin{equation}
    \lim_{\epsilon\rightarrow 0^{+}}\left (\frac{1}{2^\nu\Gamma(1+\nu)}\right )^2\int_{\R^d}\left (\hat{K}_{\nu}(|\xi|\epsilon)-\bar{C}_{\nu}(\xi,\epsilon)\right )\tilde{K}_{1+\nu}(|\xi|\epsilon)\left (|\xi|^{2\nu}\hat{f}(\xi)\hat{g}(\xi)^* \right )\,d\xi,
\end{equation}
via an application of the dominated convergence theorem. To commence, we establish uniform bounds (independent of the $\epsilon$) of the integrand. 

Observe that 
\begin{equation*}
    \left | \tilde{K}_{1+\nu}(|\xi|\epsilon)\left (|\xi|^{2\nu}\hat{f}(\xi)\hat{g}(\xi)^*\right ) \right |\lesssim \left |\left (|\xi|^{2\nu}\hat{f}(\xi)\hat{g}(\xi)^*\right ) \right |,
\end{equation*}
since $\tilde{K}_{1+\nu}$ is bounded. Since the above is integrable, to be able to apply the dominated convergence theorem, we only need to establish bounds for 
\begin{equation*}
    \left |\hat{K}_{\nu}(|\xi|\epsilon)-\bar{C}_{\nu}(\xi,\epsilon)\right |,
\end{equation*}
that are independent of $\epsilon$. We proceed by 2 cases $|\xi|\epsilon\leq 1$ and 
    $|\xi|\epsilon\geq 1$. 
\begin{list}{$\circ$}{}
    \item $|\xi|\epsilon\leq 1$\ \\
    From \eqref{eqn:renormalize_decomp}, we have that 
    \begin{equation*}
     \left |\hat{K}_{\nu}(|\xi|\epsilon)-\bar{C}_{\nu}(\xi,\epsilon)\right |\lesssim 
     \begin{cases}
         &1 + |A_{\nu}(|\xi|\epsilon)|, \quad \text{for $\nu \in \R_{+}\setminus \N_0,$}\\
         &1+|\log(|\xi|^2)|+|A_{\nu}(|\xi|\epsilon)|, \quad \text{for $\nu \in \N_0.$}
     \end{cases}
 \end{equation*}

If we show that $|A_{\nu}(|\xi|\epsilon)|$ is uniformly bounded (with respect to $\epsilon\in (0,1)$) then we can apply the dominated convergence theorem. We consider the case $\nu \in \R_{+}\setminus\N_0$ and $\nu \in \N_0$ separately.

For $\nu \in \R_{+}\setminus\N_0$, by definition of $A_{\nu}(z)$, see \eqref{eqn:Anu}, and the fact that $|\xi|\epsilon\leq 1$, we have
\begin{equation}
    |A_{\nu}(|\xi|\epsilon)|\lesssim
    \sum_{j=\lfloor \nu \rfloor +1}^\infty
\left (\frac{1}{2^{2j-\nu}j!\Gamma(j-\nu+1)}\right )+\sum_{j=1}^\infty \left (\frac{1}{2^{2j+\nu}j!\Gamma(j+\nu+1)} \right ),
\end{equation}
which is finite and independent of $\epsilon$. 

We now consider the case when $\nu \in \N_0$. Again, since $|\xi|\epsilon\leq 1$, we have that 
\begin{equation*}
\begin{split}
    |A_{\nu}(|\xi|\epsilon)|&\leq \sum_{j=1}^\infty \left (\frac{|\psi(j+1)|+|\psi(j+\nu+1)|+2\log(2)}{2^{2j+\nu}j!(\nu+j)!}\right )+\sum_{j=1}^\infty \left (\frac{1}{2^{2j+\nu-1}j!(\nu+j)!}\right ),\\
    &\leq \sum_{j=1}^\infty \left (\frac{|\log(j+1)|+|\log(j+\nu+1)|+2\log(2)}{2^{2j+\nu}j!(\nu+j)!}\right )+\sum_{j=1}^\infty \left (\frac{1}{2^{2j+\nu-1}j!(\nu+j)!}\right )
\end{split}
\end{equation*}
    which is finite and independent of $\epsilon$. In the above, we use the fact that $|\psi(z)|\leq |\log(z)|$, see \eqref{eqn:upper_lower_bnd_psi}.

    Hence, we have that 
   \begin{equation*}
     \begin{split}
              &\left |(\hat{K}_{\nu}(|\xi|\epsilon)-\bar{C}_{\nu}(\xi,\epsilon))\tilde{K}_{1+\nu}(|\xi|\epsilon)\left (|\xi|^{2\nu}\hat{f}(\xi)\hat{g}(\xi)^* \right )\right |\mathds{1}_{|\xi|\epsilon \leq 1}\\
     &\lesssim 
     \begin{cases}
         \left ||\xi|^{2\nu}\hat{f}(\xi)\hat{g}(\xi)^* \right |, \quad &\text{for $\nu \in \R_{+}\setminus \N_0,$}\\
         (1+|\log(|\xi|^2)|)\left ||\xi|^{2\nu}\hat{f}(\xi)\hat{g}(\xi)^* \right |, \quad &\text{for $\nu \in \N_0,$}
     \end{cases}
     \end{split}
 \end{equation*}
 which is integrable and independent of $\epsilon$. 
        
    \item $|\xi|\epsilon\geq 1$

Note that we have (since $0< \log(\epsilon^{-1})\leq \log|\xi|$ for $\epsilon \in (0,1)$)
\begin{equation}
\begin{split}
    \left |\hat{K}_{\nu}(|\xi|\epsilon)-\bar{C}_{\nu}(\xi,\epsilon)\right |&\leq \left |\hat{K}_{\nu}(|\xi|\epsilon)|+|\bar{C}_{\nu}(\xi,\epsilon)\right |\\
    &\lesssim 1+\left |{C}_{\nu}(|\xi|\epsilon)\right |+\log(\epsilon^{-1}),\\
        &\lesssim 1+\log(|\xi|).\\
\end{split}
\end{equation}
       Hence, we have that,  for $\nu \in \R_{+}$, 
   \begin{equation*}
     \begin{split}
              &\left |(\hat{K}_{\nu}(|\xi|\epsilon)-\bar{C}_{\nu}(\xi,\epsilon))\tilde{K}_{1+\nu}(|\xi|\epsilon)\left (|\xi|^{2\nu}\hat{f}(\xi)\hat{g}(\xi)^* \right )\right |\mathds{1}_{|\xi|\epsilon \geq 1}\\
     &\lesssim 
         \left |\log(|\xi|)|\xi|^{2\nu}\hat{f}(\xi)\hat{g}(\xi)^* \right |\mathds{1}_{|\xi|\geq 1},
         \end{split}
 \end{equation*}
which is integrable and independent of $\epsilon$. 
\end{list}
We can now apply the dominated convergence theorem. We have that 
\begin{equation}
\begin{split}
    &\lim_{\epsilon\rightarrow 0^{+}}\left (\frac{1}{2^\nu\Gamma(1+\nu)}\right )^2\int_{\R^d}\left (\hat{K}_{\nu}(|\xi|\epsilon)-\bar{C}_{\nu}(\xi,\epsilon)\right )\tilde{K}_{1+\nu}(|\xi|\epsilon)\left (|\xi|^{2\nu}\hat{f}(\xi)\hat{g}(\xi)^* \right )\,d\xi\\
    &=\left (\frac{1}{2^\nu\Gamma(1+\nu)}\right )^2\int_{\R^d}\lim_{\epsilon\rightarrow 0^{+}}\left (\hat{K}_{\nu}(|\xi|\epsilon)-\bar{C}_{\nu}(\xi,\epsilon)\right )\tilde{K}_{1+\nu}(|\xi|\epsilon)\left (|\xi|^{2\nu}\hat{f}(\xi)\hat{g}(\xi)^* \right )\,d\xi,\\
    &=\frac{1}{2^\nu\Gamma(1+\nu)}\int_{\R^d}\lim_{\epsilon\rightarrow 0^{+}}\left (\hat{K}_{\nu}(|\xi|\epsilon)-\bar{C}_{\nu}(\xi,\epsilon)\right )\left (|\xi|^{2\nu}\hat{f}(\xi)\hat{g}(\xi)^* \right )\,d\xi. 
\end{split}
\end{equation}
By considering \eqref{eqn:series_laurent}, we have proved the first and second point of Proposition \ref{prop:renormalized_lim}. 

\end{proof}

\subsection{Proof of Proposition \ref{prop:renormalized_lim_1}}\label{subsec:renormal_lim_2}
Before we directly compute the renormalized limit of the integral:

\begin{equation}
    \int_{\R^d}\tilde{C}_{\nu}(\xi,\epsilon)\tilde{K}_{1+\nu}(|\xi|\epsilon)\left (|\xi|^{2\nu}\hat{f}(\xi)\hat{g}(\xi)^* \right )\,d\xi,
\end{equation}
It will be useful to make some preliminary computations. 

Recalling \eqref{eqn:cnu_tilde}, we have that 
\begin{equation*}
\begin{split}
    &\int_{\R^d}\bar{C}_{\nu}(\xi,\epsilon)\tilde{K}_{1+\nu}(|\xi|\epsilon)\left (|\xi|^{2\nu}\hat{f}(\xi)\hat{g}(\xi)^* \right )\,d\xi\\
    &=\mathds{1}_{\nu \in \N_0}\frac{(-1)^\nu}{2^{\nu-1}\nu!}\log(\epsilon^{-1})\int_{\R^d}\tilde{K}_{1+\nu}(|\xi|\epsilon)\left (|\xi|^{2\nu}\hat{f}(\xi)\hat{g}(\xi)^* \right )\,d\xi\\
    &\qquad \qquad +\int_{\R^d}{C}_{\nu}(|\xi|\epsilon)\tilde{K}_{1+\nu}(|\xi|\epsilon)\left (|\xi|^{2\nu}\hat{f}(\xi)\hat{g}(\xi)^* \right )\,d\xi. 
\end{split}
\end{equation*}

We now compute the limit of the first integral:
\begin{equation*}
\mathds{1}_{\nu \in \N_0}\frac{(-1)^\nu}{2^{\nu-1}\nu!}\log(\epsilon^{-1})\int_{\R^d}\tilde{K}_{1+\nu}(|\xi|\epsilon)\left (|\xi|^{2\nu}\hat{f}(\xi)\hat{g}(\xi)^* \right )\,d\xi. 
\end{equation*}
We assume that $\nu \in \N_0$ without loss of generality. Observe that 
\begin{equation*}
\begin{split}
    &\log(\epsilon^{-1})\int_{\R^d}\tilde{K}_{1+\nu}(|\xi|\epsilon)\left (|\xi|^{2\nu}\hat{f}(\xi)\hat{g}(\xi)^* \right )\,d\xi\\
    &=\log(\epsilon^{-1})\int_{\R^d}\left (\tilde{K}_{1+\nu}(|\xi|\epsilon)-\tilde{K}_{1+\nu}(0)\right )\left (|\xi|^{2\nu}\hat{f}(\xi)\hat{g}(\xi)^* \right )\,d\xi\\
    &\qquad \qquad +\log(\epsilon^{-1})\tilde{K}_{1+\nu}(0)\int_{\R^d}\left (|\xi|^{2\nu}\hat{f}(\xi)\hat{g}(\xi)^* \right )\,d\xi\\
    &=\log(\epsilon^{-1})\int_{\R^d}\left (\tilde{K}_{1+\nu}(|\xi|\epsilon)-\tilde{K}_{1+\nu}(0)\right )\left (|\xi|^{2\nu}\hat{f}(\xi)\hat{g}(\xi)^* \right )\,d\xi+\text{singular terms}. 
\end{split}
\end{equation*}

Our attention now takes us to the second integral:
\begin{equation*}
    \int_{\R^d}{C}_{\nu}(|\xi|\epsilon)\tilde{K}_{1+\nu}(|\xi|\epsilon)\left (|\xi|^{2\nu}\hat{f}(\xi)\hat{g}(\xi)^* \right )\,d\xi. 
\end{equation*}
For simplicity, we will write 
\begin{equation*}
    {C}_{\nu}(z)=\mathds{1}_{\nu \ne 0}\sum_{j=0}^{\lfloor \nu \rfloor}c_{j}z^{-2(\nu-j)},\quad \text{for $z>0,$}
\end{equation*}
where the constants $c_{j}$ can be obtained from \eqref{eqn:Cnu}. 

In turn, we have that 
\begin{equation*}
\begin{split}
    &\int_{\R^d}{C}_{\nu}(|\xi|\epsilon)\tilde{K}_{1+\nu}(|\xi|\epsilon)\left (|\xi|^{2\nu}\hat{f}(\xi)\hat{g}(\xi)^* \right )\,d\xi\\
    &=\mathds{1}_{\nu \ne 0}\sum_{j=0}^{\lfloor \nu \rfloor}c_{j}\epsilon^{-2(\nu-j)}\int_{\R^d}|\xi|^{-2(\nu-j)}\tilde{K}_{1+\nu}(|\xi|\epsilon)\left (|\xi|^{2\nu}\hat{f}(\xi)\hat{g}(\xi)^* \right )\,d\xi,\\
    &=\mathds{1}_{\nu \ne 0}\sum_{j=0}^{\lfloor \nu \rfloor}c_{j}\epsilon^{-2(\nu-j)}\int_{\R^d}\tilde{K}_{1+\nu}(|\xi|\epsilon)\left (|\xi|^{2j}\hat{f}(\xi)\hat{g}(\xi)^* \right )\,d\xi,\\
    &=\mathds{1}_{\nu \ne 0}\sum_{j=0}^{\lfloor \nu \rfloor}c_{j}\epsilon^{-2\nu}\int_{\R^d}(|\xi|\epsilon )^{2j}\tilde{K}_{1+\nu}(|\xi|\epsilon)\left (\hat{f}(\xi)\hat{g}(\xi)^* \right )\,d\xi. 
\end{split}
\end{equation*}

From this, our focus will be on computing 
\begin{equation}\label{eqn:renormalized_limit_1}
    \rdash\lim_{\epsilon\rightarrow 0^{+}}\epsilon^{-2\nu}\int_{\R^d}(|\xi|\epsilon )^{2j}\tilde{K}_{1+\nu}(|\xi|\epsilon)\left (\hat{f}(\xi)\hat{g}(\xi)^* \right )\,d\xi,
\end{equation}
for $j=0,..,\lfloor \nu \rfloor$. To do this, we need to identify the singular terms that appear in $\epsilon^{-2\nu}(|\xi|\epsilon )^{2j}\tilde{K}_{1+\nu}(|\xi|\epsilon).$ We do this by separately considering the cases $\nu \in \R_{+}\setminus \N_0$ and $\nu \in \N_0.$

\begin{itemize}
\item $\nu \in \R_{+}\setminus \N_0$\ \\
We find an expression for $z^{2j}\tilde{K}_{1+\nu}(z)$, when $z=|\xi|\epsilon$.

From \eqref{eqn:series_laurent}, we have that 
	\begin{equation}
		\tilde{K}_{\nu}(z)=\frac{\pi}{2\sin(\pi\nu)}\sum_{k=0}^\infty \left ( \frac{z^{2k}}{2^{2k-\nu}k!\Gamma(k-\nu+1)} -\frac{z^{2(k+1+\nu)}}{2^{2k+\nu}k!\Gamma(k+\nu+1)}\right ), \quad \text{for $z>0$}. 
	\end{equation}
    In turn, we have that 
\begin{equation*}
\begin{split}
            &\epsilon^{-2\nu}(|\xi|\epsilon)^{2j}\tilde{K}_{\nu}(|\xi|\epsilon)\\
            &=\frac{\pi}{2\sin(\pi\nu)}\sum_{k=0}^\infty \epsilon^{2(k+j-\nu)}\left (\frac{|\xi|^{2(k+j)}}{2^{2k-\nu}k!\Gamma(k-\nu+1)}\right)\\
        &\qquad\qquad  -\frac{\pi}{2\sin(\pi\nu)}\sum_{k=0}^\infty
        \epsilon^{2(k+j+1)}\left (\frac{|\xi|^{2(k+j+1+\nu)}}{2^{2k+\nu}k!\Gamma(k+\nu+1)}\right ). 
\end{split}
\end{equation*}
        By identifying the singular terms, we have that 
        \begin{equation}\label{eqn:singular_terms_D}
        \begin{split}
        &\epsilon^{-2\nu}(|\xi|\epsilon)^{2j}\tilde{K}_{\nu}(|\xi|\epsilon)\\
        &=\frac{\pi}{2\sin(\pi\nu)}\sum_{k=0}^{\lfloor \nu\rfloor-j} \epsilon^{2(k+j-\nu)}\left (\frac{|\xi|^{2(k+j)}}{2^{2k-\nu}k!\Gamma(k-\nu+1)}\right)\\
        &\qquad +\frac{\pi}{2\sin(\pi\nu)}\sum_{k=\lfloor \nu\rfloor-j+1}^{\infty} \epsilon^{2(k+j-\nu)}\left (\frac{|\xi|^{2(k+j)}}{2^{2k-\nu}k!\Gamma(k-\nu+1)}\right)\\
        &\qquad \qquad  -\frac{\pi}{2\sin(\pi\nu)}\sum_{k=0}^\infty
        \epsilon^{2(k+j+1)}\left (\frac{|\xi|^{2(k+j+1+\nu)}}{2^{2k+\nu}k!\Gamma(k+\nu+1)}\right ),\\
        &=\frac{\pi}{2\sin(\pi\nu)}\sum_{k=0}^{\lfloor \nu\rfloor-j} \epsilon^{2(k+j-\nu)}\left (\frac{|\xi|^{2(k+j)}}{2^{2k-\nu}k!\Gamma(k-\nu+1)}\right)\\
        &\qquad +\frac{\pi}{2\sin(\pi\nu)}|\xi|^{2\nu}\sum_{k=\lfloor \nu\rfloor-j+1}^{\infty} \left (\frac{(|\xi|\epsilon)^{2(k+j-\nu)}}{2^{2k-\nu}k!\Gamma(k-\nu+1)}\right)\\
        &\qquad \qquad  -\frac{\pi}{2\sin(\pi\nu)}|\xi|^{2\nu}\sum_{k=0}^\infty
        \left (\frac{(|\xi|\epsilon)^{2(k+j+1)}}{2^{2k+\nu}k!\Gamma(k+\nu+1)}\right ),\\
        \end{split}
    \end{equation}
Hence, the singular terms are 
\begin{equation*}
    D_{\nu}(\xi,\epsilon):=\frac{\pi}{2\sin(\pi\nu)}\sum_{k=0}^{\lfloor \nu\rfloor-j} \epsilon^{2(k+j-\nu)}\left (\frac{|\xi|^{2(k+j)}}{2^{2k-\nu}k!\Gamma(k-\nu+1)}\right). 
\end{equation*}

In order to compute \eqref{eqn:renormalized_limit_1}, we consider the following:
\begin{equation*}
\begin{split}
    &\epsilon^{-2\nu}\int_{\R^d}(|\xi|\epsilon )^{2j}\tilde{K}_{1+\nu}(|\xi|\epsilon)\left (\hat{f}(\xi)\hat{g}(\xi)^* \right )\,d\xi\\
    &=\int_{\R^d}\left (\epsilon^{-2\nu}(|\xi|\epsilon )^{2j}\tilde{K}_{1+\nu}(|\xi|\epsilon) -D_{\nu}(\xi,\epsilon)\right )\left (\hat{f}(\xi)\hat{g}(\xi)^* \right )\,d\xi\\
    &\qquad +\int_{\R^d}D_{\nu}(\xi,\epsilon)\left (\hat{f}(\xi)\hat{g}(\xi)^* \right )\,d\xi,\\
    &=\int_{\R^d}\left (\epsilon^{-2\nu}(|\xi|\epsilon )^{2j}\tilde{K}_{1+\nu}(|\xi|\epsilon) -D_{\nu}(\xi,\epsilon)\right )\left (\hat{f}(\xi)\hat{g}(\xi)^* \right )\,d\xi\\
    &\qquad +\int_{\R^d}D_{\nu}(\xi,\epsilon)\left (\hat{f}(\xi)\hat{g}(\xi)^* \right )\,d\xi,\\
    &=\int_{\R^d}\left (\epsilon^{-2\nu}(|\xi|\epsilon )^{2j}\tilde{K}_{1+\nu}(|\xi|\epsilon) -D_{\nu}(\xi,\epsilon)\right )\left (\hat{f}(\xi)\hat{g}(\xi)^* \right )\,d\xi+\text{singular terms}.
\end{split}
\end{equation*}
Hence, to obtain the value of \eqref{eqn:renormalized_limit_1}, we need to compute 
\begin{equation}
    \lim_{\epsilon\rightarrow 0^{+}}\int_{\R^d}\left (\epsilon^{-2\nu}(|\xi|\epsilon )^{2j}\tilde{K}_{1+\nu}(|\xi|\epsilon) -D_{\nu}(\xi,\epsilon)\right )\left (\hat{f}(\xi)\hat{g}(\xi)^* \right )\,d\xi. 
\end{equation}
We proceed via an application of the dominated convergence theorem. We now establish bounds of the integrand that are independent of $\epsilon\in(0,1).$ We consider the cases $|\xi|\epsilon\leq 1$ and $|\xi|\epsilon \geq 1$ separately. 

\begin{itemize}
    \item We focus on establishing a uniform bound for the case $|\xi|\epsilon\leq 1$.

From \eqref{eqn:singular_terms_D} and the assumption that $|\xi|\epsilon\leq 1$, we have that 
\begin{equation*}
\begin{split}
    &\left |\left (\epsilon^{-2\nu}(|\xi|\epsilon )^{2j}\tilde{K}_{1+\nu}(|\xi|\epsilon) -D_{\nu}(\xi,\epsilon)\right )\left (\hat{f}(\xi)\hat{g}(\xi)^* \right ) \right |\\
    &\lesssim \left ||\xi|^{2\nu}\left (\hat{f}(\xi)\hat{g}(\xi)^* \right ) \right |\sum_{k=\lfloor \nu\rfloor-j+1}^{\infty} \left (\frac{(|\xi|\epsilon)^{2(k+j-\nu)}}{2^{2k-\nu}k!\Gamma(k-\nu+1)}\right)\\
        &\qquad \qquad  +\left ||\xi|^{2\nu}\left (\hat{f}(\xi)\hat{g}(\xi)^* \right ) \right |\sum_{k=0}^\infty
        \left (\frac{(|\xi|\epsilon)^{2(k+j+1)}}{2^{2k+\nu}k!\Gamma(k+\nu+1)}\right ),\\
    &\lesssim \left ||\xi|^{2\nu}\left (\hat{f}(\xi)\hat{g}(\xi)^* \right ) \right |\sum_{k=\lfloor \nu\rfloor-j+1}^{\infty} \left (\frac{1}{2^{2k-\nu}k!\Gamma(k-\nu+1)}\right)\\
        &\qquad \qquad  +\left ||\xi|^{2\nu}\left (\hat{f}(\xi)\hat{g}(\xi)^* \right ) \right |\sum_{k=0}^\infty
        \left (\frac{1}{2^{2k+\nu}k!\Gamma(k+\nu+1)}\right ).
\end{split}
\end{equation*}
In turn, we have that 
\begin{equation}
    \left |\left (\epsilon^{-2\nu}(|\xi|\epsilon )^{2j}\tilde{K}_{1+\nu}(|\xi|\epsilon) -D_{\nu}(\xi,\epsilon)\right )\left (\hat{f}(\xi)\hat{g}(\xi)^* \right ) \right |\mathds{1}_{|\xi|\leq \epsilon^{-1}}\lesssim \left ||\xi|^{2\nu}\left (\hat{f}(\xi)\hat{g}(\xi)^* \right ) \right |,
\end{equation}
which is integrable and independent of $\epsilon \in (0,1).$
    \item We focus our attention on the case $|\xi|\epsilon \geq 1. $
        Note that, since $\tilde{K}_{1+\nu}$ decays at an exponential rate, we have that 
        \begin{equation*}
            |(|\xi|\epsilon )^{2j}\tilde{K}_{1+\nu}(|\xi|\epsilon)|\lesssim 1. 
        \end{equation*}

        Moreover, since $\epsilon^{-1}\leq |\xi|$, we have that 
        \begin{equation*}
    |D_{\nu}(\xi,\epsilon)|\lesssim \sum_{k=0}^{\lfloor \nu\rfloor-j} |\xi|^{4(k+j)-2\nu}. 
\end{equation*}
Hence, we have that 
\begin{equation*}
    \begin{split}
        &\left |\left (\epsilon^{-2\nu}(|\xi|\epsilon )^{2j}\tilde{K}_{1+\nu}(|\xi|\epsilon) -D_{\nu}(\xi,\epsilon)\right )\left (\hat{f}(\xi)\hat{g}(\xi)^* \right ) \right |\mathds{1}_{\epsilon^{-1}\leq |\xi|}\\
    &\lesssim \sum_{k=0}^{\lfloor \nu\rfloor-j}\left ||\xi|^{4(k+j)}\left (\hat{f}(\xi)\hat{g}(\xi)^* \right ) \right |\mathds{1}_{1\leq |\xi|}.
    \end{split}
\end{equation*}
\end{itemize}
The previous 2 points show that we can apply the dominated convergence theorem. In turn, this gives us that 
\begin{equation}
    \lim_{\epsilon\rightarrow 0^{+}}\int_{\R^d}\left (\epsilon^{-2\nu}(|\xi|\epsilon )^{2j}\tilde{K}_{1+\nu}(|\xi|\epsilon) -D_{\nu}(\xi,\epsilon)\right )\left (\hat{f}(\xi)\hat{g}(\xi)^* \right )\,d\xi=0, 
\end{equation}
for $j=0,..,\lfloor \nu \rfloor$. Hence, we have proved the result for the case when $\nu \in \R\setminus \N_0.$

\item $\nu \in \N_0$ \ \\
 Like in the previous case, we want to find an expression for $\epsilon^{-2\nu}(|\xi|\epsilon )^{2j}\tilde{K}_{1+\nu}(|\xi|\epsilon),$ for when $\nu \in \N_0$. Using a characterization of ${K}_{1+\nu}$ that appears in \cite[Section 3.2]{derezi2023generalized} we have the following expression 
\begin{equation*}
    \begin{split}
        z^{2j}\tilde{K}_{1+\nu}(z)&=\sum_{k=0}^{\nu}\frac{(-1)^k(\nu-k)!}{2^{2k-\nu}k!}z^{2(k+j)}\\
    &+(-1)^\nu\sum_{k=0}^\infty \left ( \frac{\psi(k+1)+\psi(k+\nu+2)+2\log(2)-2\log(z)}{2^{2k+\nu}k!(k+\nu)!}\right )z^{2(k+j+1+\nu)},
    \end{split}
\end{equation*}
for $z>0.$ In turn, we have that 
\begin{equation}\label{eqn:integer_sum}
\begin{split}
    &\epsilon^{-2\nu}(|\xi|\epsilon )^{2j}\tilde{K}_{1+\nu}(|\xi|\epsilon)\\
    &=\mathds{1}_{\nu\ne j}\sum_{k=0}^{\nu -j-1}\frac{(-1)^k(\nu-k)!}{2^{2k-\nu}k!}\epsilon^{-\left (\nu-j-k\right )}|\xi|^{2(j+k)}\\
    &+|\xi|^{2\nu}\sum_{k=\nu -j-1}^{\nu}\frac{(-1)^k(\nu-k)!}{2^{2k-\nu}k!}(\epsilon|\xi|)^{-\left (\nu-j-k\right )}\\
    &+(-1)^\nu|\xi|^{2\nu}\sum_{k=0}^\infty E(\xi,\epsilon,k)(\epsilon|\xi|)^{2(k+j+1)}\\
    &+(-1)^{\nu+1}|\xi|^{2\nu}\sum_{k=0}^\infty \frac{1}{2^{2k+\nu-1}k!(k+\nu)!}\log(\epsilon|\xi|)(\epsilon|\xi|)^{2(k+j+1)},
\end{split}
\end{equation}
where 
\begin{equation*}
    E(\xi,\epsilon,k):=\frac{\psi(k+1)+\psi(k+\nu+2)+2\log(2)}{2^{2k+\nu}k!(k+\nu)!}. 
\end{equation*}

Hence, the singular terms are 
\begin{equation*}
    D_{\nu}(\xi,\epsilon):=\mathds{1}_{\nu\ne j}\sum_{k=0}^{\nu -j-1}\frac{(-1)^k(\nu-k)!}{2^{2k-\nu}k!}\epsilon^{-\left (\nu-j-k\right )}|\xi|^{2(j+k)}. 
\end{equation*}
Again, like in the case when $\nu \in \R_{+}\setminus\N_0$, we have that 
\begin{equation*}
\begin{split}
    &\epsilon^{-2\nu}\int_{\R^d}(|\xi|\epsilon )^{2j}\tilde{K}_{1+\nu}(|\xi|\epsilon)\left (\hat{f}(\xi)\hat{g}(\xi)^* \right )\,d\xi\\
    &=\int_{\R^d}\left (\epsilon^{-2\nu}(|\xi|\epsilon )^{2j}\tilde{K}_{1+\nu}(|\xi|\epsilon) -D_{\nu}(\xi,\epsilon)\right )\left (\hat{f}(\xi)\hat{g}(\xi)^* \right )\,d\xi+\text{singular terms}.
\end{split}
\end{equation*}
Hence, in order to compute the renormalized limit, we need to compute 
\begin{equation*}
    \lim_{\epsilon\rightarrow 0^{+}}\int_{\R^d}\left (\epsilon^{-2\nu}(|\xi|\epsilon )^{2j}\tilde{K}_{1+\nu}(|\xi|\epsilon) -D_{\nu}(\xi,\epsilon)\right )\left (\hat{f}(\xi)\hat{g}(\xi)^* \right )\,d\xi.
\end{equation*}
Like in the case when $\nu \in \R_{+}\setminus \N_0$ this can be done via an application of the dominated convergence theorem. The same arguments lead to the 
\begin{equation*}
    \lim_{\epsilon\rightarrow 0^{+}}\int_{\R^d}\left (\epsilon^{-2\nu}(|\xi|\epsilon )^{2j}\tilde{K}_{1+\nu}(|\xi|\epsilon) -D_{\nu}(\xi,\epsilon)\right )\left (\hat{f}(\xi)\hat{g}(\xi)^* \right )\,d\xi=0.
\end{equation*}

\end{itemize}
Overall, we have that 
\begin{equation}
    \rdash\lim_{\epsilon\rightarrow 0^{+}}\epsilon^{-2\nu}\int_{\R^d}(|\xi|\epsilon )^{2j}\tilde{K}_{1+\nu}(|\xi|\epsilon)\left (\hat{f}(\xi)\hat{g}(\xi)^* \right )\,d\xi=0,
\end{equation}
which finishes our proof. 

\section*{Acknowledgements}

The author would like to thank Antoine Gloria, Félix del Teso, Huyuan Chen, Tobias Weth, Daniel Hauer for useful discussions in preparation of this note. Moreover, the author is especially grateful to Alberto Bonicelli and Lorenzo Zambotti for their insights and helping me navigate the mathematics behind Quantum Field Theory.

Moreover, the author acknowledges financial support from
the European Research Council (ERC) under the European Union’s Horizon 2020 research
and innovation programme (Grant Agreement n◦864066). Views and opinions expressed
are however those of the author only and do not necessarily reflect those of the European
Union or the European Research Council Executive Agency. Neither the European Union
nor the granting authority can be held responsible for them.

\bibliographystyle{unsrt}
\bibliography{\jobname}

\end{document}